\documentclass[12pt]{amsart}
\usepackage[margin=1in]{geometry}
\usepackage{mathptmx}
\usepackage{mathtools}
\usepackage{amsthm}
\usepackage{amssymb}
\usepackage[alphabetic]{amsrefs}
\usepackage{graphicx}
\usepackage{tikz}
	\usetikzlibrary{decorations.pathreplacing}
	\usetikzlibrary{patterns}
\usepackage{caption}
\usepackage{euscript}
\usepackage{flexisym}

\newcommand{\ra}{\rightarrow}
\newcommand{\pr}{\prime}
\newcommand{\de}{\partial}
\newcommand{\te}{\theta}

\newcommand{\N}{\mathbb{N}}

\newcommand{\R}{\mathbb{R}}
\newcommand{\Z}{\mathbb{Z}}

\newcommand{\abs}[1]{\left\lvert #1 \right\rvert}
\newcommand{\tld}[1]{\widetilde{#1}}
\newcommand{\lbar}[1]{\overline{#1}}
\newcommand{\ubar}[1]{\underline{#1}}
\newcommand{\id}{\mathrm{id}}
\DeclareMathOperator{\homeo}{Homeo}

\renewcommand{\coprod}{\rotatebox[origin = c]{180}{$\prod$}}
\newcommand\independent{\protect\mathpalette{\protect\independenT}{\perp}}
\def\independenT#1#2{\mathrel{\rlap{$#1#2$}\mkern2mu{#1#2}}}

\newcommand{\hookuparrow}{\mathrel{\rotatebox[origin=c]{90}{$\hookrightarrow$}}}
\newcommand{\tikzmark}[1]{\tikz[overlay,remember picture] \node (#1) {};}
\tikzset{square arrow/.style={to path={-- ++(0,-.25) -| (\tikztotarget)}}}
\newcommand{\lgast}{\scalebox{1.5}{$\ast$}}
\newcommand{\nsharp}{\tikz[anchor=base,baseline]{\draw (0,0.4) -- (0,0.05) -- (0.2,0.05) -- (0.2,-0.12); \draw (0,0.25) -- (0.2,0.25) -- (0.2,0.05);}}
\newcommand{\sslash}{\mathbin{/\mkern-6mu/}}

\newtheorem{thm}{Theorem}

\newtheorem{lemma}{Lemma}
\newtheorem{prop}{Proposition}
\newtheorem{conj}{Conjecture}
\theoremstyle{definition}
\newtheorem{defin}{Definition}
\newtheorem*{defin*}{Definition}
\newtheorem*{ext}{Extension}
\theoremstyle{remark}
\newtheorem*{note}{Note}
\newtheorem*{warning}{Warning}
\newtheorem*{ex}{Example}

\begin{document}

\title{Controlled Mather-Thurston Theorems}

\author{Michael Freedman}
\address{\hskip-\parindent
	Michael Freedman \\
	Microsoft Research, Station Q, and Department of Mathematics \\
	University of California, Santa Barbara \\
	Santa Barbara, CA 93106 \\
}

\begin{abstract}
    Classical results of Milnor, Wood, Mather, and Thurston produce flat connections in surprising places. The Milnor-Wood inequality is for circle bundles over surfaces, whereas the Mather-Thurston Theorem is about cobording general manifold bundles to ones admitting a flat connection. The surprise comes from the close encounter with obstructions from Chern-Weil theory and other smooth obstructions such as the Bott classes and the Godbillion-Vey invariant. Contradiction is avoided because the structure groups for the positive results are larger than required for the obstructions, e.g.\ $\operatorname{PSL}(2,\R)$ versus $\operatorname{U}(1)$ in the former case and $C^1$ versus $C^2$ in the latter. This paper adds two types of control strengthening the positive results: In many cases we are able to (1) refine the Mather-Thurston cobordism to a semi-$s$-cobordism (ssc) and (2) provide detail about how, and to what extent, transition functions must wander from an initial, small, structure group into a larger one.

    The motivation is to lay mathematical foundations for a physical program. The philosophy is that living in the IR we cannot expect to know, for a given bundle, if it has curvature or is flat, because we can't resolve the fine scale topology which may be present in the base, introduced by a ssc, nor minute symmetry violating distortions of the fiber. Small scale, UV, ``distortions" of the base topology and structure group allow flat connections to simulate curvature at larger scales. The goal is to find a duality under which curvature terms, such as Maxwell's $F \wedge F^\ast$ and Hilbert's $\int R\ dvol$ are replaced by an action which measures such ``distortions." In this view, curvature results from renormalizing a discrete, group theoretic, structure.
\end{abstract}

\maketitle

\section{Introduction}

Let us recall two sources of inspiration. Milnor and Wood \cites{mil58,wood} considered circle bundles
\begin{tikzpicture}
    \node at (0,0) {$S^1$};
    \draw [->] (0.2,0) -- (0.8,0);
    \node at (1.2,0) {$M^3$};
    \draw [->] (1.2,-0.3) -- (1.2,-0.9);
    \node at (1.2,-1.2) {$\Sigma^2$};
    \node at (-0.4,-1.2) {$B =$};
\end{tikzpicture}
over a closed surface and studied the relationship between the Euler class $\chi(B)$ and the existence of a flat connection for different choices of structure group $G$. When $G = \mathrm{U}(1)$ i.e.\ $B$ is a principle bundle it is well known that there is an integral formula for $\chi(B)$ in terms of the curvature form $\Omega$
\begin{equation}
    \chi(B) = \frac{1}{2\pi} \int_\Sigma \mathrm{Pfaffian}(\Omega)
\end{equation}
So for $G = U(1)$, $B$ has a flat connection iff $\chi(B) = 0$.

On the other hand if $G = \mathrm{PSL}(2,\R)$ and represents on $S^1$ via its action on the circle at infinity of $H^2$, the hyperbolic plane, then $B$ has a flat connection iff
\begin{equation}\label{ineq}
    \abs{\chi(B)} \leq \abs{\chi(\tau(\Sigma))} = \abs{\chi(\Sigma)}
\end{equation}
where $\tau$ denotes the unit tangent bundle.

In the equality case, $B \cong \tau(\Sigma)$, the unit tangent bundle, the local trivialization induced by the flat connection identifies tangent spaces at nearby points $x,y \in H^2$ by using geodesic flow to match the unit spheres at both $x$ and $y$ with the ideal circle at infinity.

Wood then showed that (\ref{ineq}) also holds even if the structure group is relaxed all the way to $\homeo^+(S^1)$, the group of orientation preserving homeomorphisms of the circle.

Similarly, but in a vastly more general context the Mather-Thurston Theorem \cite{thur74a} states that for any smooth manifold $X$ (compact, noncompact, bounded, no boundary, of any finite dimension) the natural map
\begin{equation}\label{homeomap}
    B\homeo^\delta(X) \ra B\homeo(X)
\end{equation}
induced by $\id: \homeo^\delta(X) \ra \homeo(X)$ from discrete to compact open topology is an acyclic map and in particular induces an isomorphism on homology $H_\ast(;\Z)$. By the Atiyah-Hirzebruch spectral sequence (\ref{homeomap}) also induces an isomorphism on bordism so the geometrically minded statement, in lowest regularity $C^0$, of the MT theorem is as follows.

Assume all manifolds have smooth structure.\footnote{Manifolds will always be finite dimensional and oriented.} Let $V^p$ be a $p$-dimensional manifold, possibly with boundary, and $X^q$ a $q$-dimensional manifold (also may have boundary) and
\begin{tikzpicture}[anchor=base, baseline]
    \node at (0,1.1) {$X$};
    \draw [->] (0.2,1.2) -- (0.8,1.2);
    \node at (1.1,1.1) {$B$};
    \draw [->] (1.1,1) -- (1.1,0.4);
    \node at (1.1,0) {$V$};
\end{tikzpicture}
a fiber bundle with transition functions, i.e.\ structure group $= \homeo(X)$. The low differentiability form of MT we consider is:

\begin{thm}[Mather-Thurston \cite{thur74a}]\label{thm:mt}
    Suppose $B$ possesses a $C^0$-transverse foliation $\mathcal{F}_0$ over a neighborhood $\EuScript{N}_1 \de V$. Then there is a cobordism $(W;V;V^\ast)$ from $V$ to $V^\ast$, constant near $\de V$ and covered by cobordism of bundles $(\lbar{B}; B, B^\ast)$, so that $B^\ast$ possesses a $C^0$-transverse foliation $\mathcal{F}$ agreeing with $\mathcal{F}_0$ on a smaller $\EuScript{N}_0(\de V) \subset \EuScript{N}_1(\de V)$. If $B$ has structure group $\homeo_0(X)$ we may arrange that this is also the structure group of $\lbar{B}$.
\end{thm}

\begin{defin}
    We treat as synonymous ``$C^0$-transverse foliation,'' ``topological transverse foliation,'' and ``topoloigcally flat connection.''
\end{defin}

Later, we would like to strengthen certain $C^0$ statements to a bilipschitz category to meet work \cite{mei} of Meigniez which requires some regularity. I would like to thank Sam Nariman for pointing out to me certain subtleties and literature gaps around bilipschitz foliations. The author recalls assertions of W.\ Thurston, during lectures delivered in the 1970s, that the individual leaves of bilip foliations can be locally smoothed, regardless of codimension or dimension. It is possible that he was working in a category, sometimes called ``tame bilipschitz,'' where all transition functions are assumed to be compositions of bilipschitz functions with bilipschitz constant arbitrarily close to one, where the proof looks much easier. (It is open if this is the in fact any restriction at all).

\begin{note}[Note on the bilipschitz category]
    In the discussion immediately following Theorem 2 of \cite{thur74b}, Thurston refers to work of his, apparently never published, that (tame?) bilipschitz foliations of any codimension $q>1$ on a smooth manifold may be (topologically) isotoped so that each compact portion of each leaf is actually a $C^\infty$-submanifold. Later \cite{c01}, Calegari proved such a smoothing statement for 2D foliations of 3-manifolds. When this principle holds it is only the separation between leaves, as seen in holonomies, that might not be differentiable, and not the leaves themselves, so such foliations are not very ``wild'' looking.

    From Thurston’s published work on the homotopy theory of foliations, \cites{thur74a,thur74b,thur76}, it is clear that he obtained that: Every microfoliation which is bilip regular and whose underlying microbundle is smooth can be made leafwise smooth by a concordance of microfoliations. We need only this coherence property. In fact, his methods also yield leafwise smoothness after a bilip isotopy bilip close to the identity. Although tameness is not essential for this part of Thurston’s work, all the bilip homeomorphisms we introduce come from representations of $\tld{\operatorname{PSL}}(2,\R)$, and are tame.
\end{note}

The original proofs of the M-T theorem hold if ``Homeo'' and ``$C^0$'' are both replaced by ``bilipschitz.'' Subsequently, Tsuboi \cite{tsu85} showed that the MT theorem also holds replacing $\homeo(X)$ with once differentiable diffeomorphisms, $C^1(X) \coloneqq \operatorname{Diff}^1(X)$, but this is more subtle. Possibly, we could also work in that context, but do not know. We work in the intermediate bilipschitz category, the most natural one for the constructions. We will not change notation but from here forward ``Homeo" can always be assumed to be bilipschitz i.e.\ all statements involving ``$C^0$'' or ``$\homeo$'' can be read literally or with Homeo and $C^0$ everywhere replaced by bilipschitz. If we explicitly state ``bilipschitz'' it is to highlight that property.

We will explore, in several contexts, adding control to the classical Mather-Thurston theorem, on cobordisms of a general bundle
\begin{tikzpicture}[anchor=base, baseline]
    \node at (0,1.1) {$X$};
    \draw [->] (0.2,1.2) -- (0.8,1.2);
    \node at (1.1,1.1) {$B$};
    \draw [->] (1.1,1) -- (1.1,0.4);
    \node at (1.1,0) {$V$};
\end{tikzpicture},
to a flat bundle
\begin{tikzpicture}[anchor=base, baseline]
    \node at (0,1.1) {$X$};
    \draw [->] (0.2,1.2) -- (0.8,1.2);
    \node at (1.1,1.1) {$B^\ast$};
    \draw [->] (1,1) -- (1,0.4);
    \node at (1.1,0) {$V^\ast$};
\end{tikzpicture}
. The first refinement is that the cobordism $(W;V,V^\ast)$, over the 3-dimensional base $V$, can always be taken to be a semi-$s$-cobordism, defined below. In higher dimensions the same statement is true after a stabilization by any product of spheres $\prod_{j=1}^J S^{i_j}$, $i_j \geq 1$, with sufficiently many factors $J$. The individual factors can all be chosen as circles or spheres of any dimension.

These results depend on the model which we call $C_1$ admitting an interesting $\pi_1$-representation in the bilipschitz category, so they \emph{only} apply up to that degree of differentiability. Possibly, following \cite{tsu85} this might be pushed to class $C^1$, but this is presently unclear.

One simplification in working with differentiability $\leq$ class $C^1$ is that the classifying space for bundles and bundles with Haefliger structures\footnote{In this context, a Haefliger structure is a $C^0$-foliation on $\tau$, the bundle over $B$ of ``vertical'' tangents, i.e.\ tangents to the fibers of $X$, which is transverse to the fibers of $\tau$. It is a theorem of Kister \cite{kister} that topological microbundles are represented by genuine $C^0$-category bundles so there is no essential novelty in defining the space $\Gamma_d^r(X)$ for $r=0$ in the usual manner. We may fix $\Gamma_d^0(X) \coloneqq \operatorname{Bun}(TX,\nu^\ast(\gamma))\sslash\operatorname{Homeo}(X)$; i.e.\ the Borel homotopy quotient of bundle maps from tangents to $X$ to the bundle $\nu^\ast(\gamma)$, $\nu: B \Gamma_d^0 \ra BGL_d$ Haefliger's map and $\gamma$ the universal bundle over $BGL_d$.} (that is on the bundle of vertical tangents $\tau$ to the original bundle) are homotopy equivalent:
\begin{equation}\label{eq:homotopyfiber}
    \ast \simeq \lbar{\Gamma}^0_d(X) \ra \Gamma^0_d(X) \xrightarrow{\simeq} B \homeo(X)
\end{equation}
i.e.\ the homotopy fiber is contractible.

Line (\ref{eq:homotopyfiber}) is most familiar for the fiber space $X = \R^d$ but remains true inserting any manifold fiber $X$ \cite{thur74a}, as explained in the footnote. The homotopy equivalence means that although we draw heavily on the theory of foliation, chiefly \cite{mei}, we only encounter Haefliger structures within certain proofs and not in theorem statements.

Theorems \ref{thm:extend1}, \ref{thm:extend2}, and \ref{thm:four} are geometric refinements of special cases of M-T.

Before stating our theorems we need to define a highly refined, directional, type of cobordism called a semi-$s$-cobordism (ssc).

\begin{defin}[Definition/Discussion]
    An absolute semi-$s$-cobordism (ssc) is a manifold triple $(W^{d+1};V,V^\ast)$ with $\de W = V \independent -V^\ast$, so that the inclusion $\operatorname{inc}: V \hookrightarrow W$ is a simple homotopy equivalence. This implies there is a simple deformation retraction $r: W \ra V$ (so that $r \circ \mathrm{inc}_V = \id_V$). But, crucially, we do not assume $V^\ast \hookrightarrow W$ is a homotopy equivalence. In the relative case we allow a common boundary $\de V = \de V^\ast$, near which $W$ is assumed to be a product.
\end{defin}

Because $r$ is a (simple) homotopy equivalence all the relative groups $H^\ast(W,V;Z[\pi_1(W)]) \cong 0$, so by Lefschetz duality $H_\ast(W,V^\ast;Z[\pi_1(W)]) \cong 0$ so the map $r \circ \mathrm{inc}_{V^\ast}: V^{\ast} \ra V$ is ``almost" itself a homotopy equivalence---it would be by Whitehead's Theorem if $(r \circ \mathrm{inc}_{V^\ast})_\#:\pi_1(V^\ast) \mapsto \pi_1(V)$ were an isomorphism.\footnote{Note that $r \circ \mathrm{inc}_{V^\ast}$ is a degree one map, so $(r \circ \mathrm{inc}_{V^\ast})_\#$ is surjective.} As it is we have an exact sequence
\begin{equation}\label{eq:sequence}
    1 \ra P \ra \pi_1(V^\ast) \ra \pi_1(V) \ra 1
\end{equation}
where $H_1(W,V^\ast;Z[\pi_1(V)]) \cong 0$ implies that $P$ is a perfect group.

In this context Quillen's plus construction \cite{dq} (which attaches an equal number of 2- and 3-cells to kill a perfect group while leaving homology unchanged), when applied to the degree one map $r \circ \mathrm{inc}_{V^\ast}: V^\ast \ra V$ yields $V^{\ast+}$ homotopy equivalent to $V$. In fact, $V^{\ast+} \overset{s}{\simeq} V$ since the vanishing of $\mathrm{Wh}(W,V)$ implies the vanishing of the Whitehead obstruction $\mathrm{Wh}(W,V^\ast)$ by duality \cite{mil71}.

A comment which will soon be useful is that any concatenation of ssc
\[
    (W; V_1,V_l)^\ast = (W_1;V_1,V_1^\ast) \bigcup_{V_1^\ast = V_2} (W_2;V_2,V_2^\ast) \bigcup_{V_2^\ast = V_3} \cdots \bigcup (W_l; V_l, V_l^\ast)
\]
is itself a ssc.

The simplest example of an ssc is obtained by taking a compact contractible manifold $K^n$ and deleting a ball from its interior; $W = \lbar{K^n \setminus B^n}$, with $\de B = V$ and $\de K = V^\ast$.

\begin{warning}
    It is tempting to denote the Quillen construction by $V^{\ast-} \overset{s}{\simeq} V$, rather than $V^{\ast+} \overset{s}{\simeq} V$, because the left boundary $V$ of $W$ is the simpler space, whereas the right boundary $V^\ast$ has been enhanced by a perfect extension. This notation would help keep straight the simple and more elaborate boundaries of $W$. Also in the context of manifolds of dimension $\geq 5$ Quillen's construction can be acomplished by an equal number of 1- and 2-surgeries; in that context cells are not being added, removing that justification for the $+$ notation. But such a change seems on par with changing the sign of the electron (clearly too late); so we stick with the $+$ sign.
\end{warning}

The innovation in Theorem \ref{thm:extend1} below is that the assumed cobordism of the base is ``directional.''  There is the problem end $V$, which a physicist would think of as the IR manifold, and a solution end $V^\ast$, the UV end, in which the fundamental group has been elaborated by a perfect extension over which the solution becomes possible. We hope\footnote{I thank Yasha Eliashberg for this suggestion, particularly in the context of the Madsen-Weiss theorem, and Sam Nariman for the suggestion of applicability to the $c$-principles of \cites{fuchs74,vassi89}.} and expect that this will turn out to be a hidden feature of most \emph{$c$-principle} theorems where homotopy ($h$-principle) by itself is inadequate, but the obstructing singularities can be removed by surgery, i.e.\ by cobordism.

Here is the new idea on the $c$-principle. Holonomy around the boundary of a 2-disk $D^2$ prevents extension of a flat connection over the disk. Traditionally, this has been dealt with by replacing $D^2$ with $S_g$, a surface of high genus and a single boundary component, a radical change of topology. When the base dimension is $\geq 3$, it is often possible to \emph{hide} the additional genus inside a semi-$s$-cobordant manifold, requiring only a subtle change in topology.

This paper came together during the isolation of COVID-19. But I always felt surrounded by good friends. I acknowledge with pleasure the mathematicians and physicists who generously shared their thoughts and insights into this project: foremost Sam Nariman for much useful feedback, as well as Parsa Bonderson, Adam Brown, Danny Calegari, Bob Edwards, Yasha Eliashberg, Slava Krushkul, Roman Lutchyn, Chaitanya Murthy, Gael Meigniez, Cliff Taubes, and Shmuel Weinberg, as well as anonymous referees.

\section{Theorems, Statements, and Conjectures/Questions}
We begin with statements that add control of the bordism $W$ beyond the usual MT Theorem. Theorem \ref{thm:extend1} is actually a special case of theorem \ref{thm:extend2} but we state it separately since its proof uses the proof of MT only as a \emph{black box}. Theorem \ref{thm:extend2} is more general but requires looking \emph{inside} Meigniez's proof of MT at a high level of detail. Theorem \ref{thm:four}, also with a black box proof, provides in a sense one half of the desired (ssc) information on $W$; the retraction $r$ but not the deformation.

Theorem \ref{thm:refine1} assumes that the data on $\de V$ is of a highly refined nature, a flat Lie-connection on a $G$-principal bundle, and makes a homological assumption needed to build a ssc to a flat Lie-connection. Theorem \ref{thm:refine2} studies the implications for Milnor-Wood when the representation $\rho: \pi_1(\Sigma) \ra \operatorname{PSL}(2,\R)$ takes generators close to $\operatorname{U}(1)$.

\begin{thm}\label{thm:extend1}
    Suppose $\dim(V) = 3$ and that the fiber bundle $B =$
    \begin{tikzpicture}[anchor=base, baseline]
        \node at (0,1.1) {$X$};
        \draw [->] (0.2,1.2) -- (0.8,1.2);
        \node at (1.1,1.1) {$E$};
        \draw [->] (1.1,1) -- (1.1,0.4);
        \node at (1.1,0) {$V$};
    \end{tikzpicture}
    has structure group $\homeo_0(X)$ and a $C^0$-transverse foliation $\mathcal{F}_0$ over $\EuScript{N}_1(\de V)$, a neighborhood of boundary $V$. Then there is a ssc $(W;V, V^\ast)$ from $V$ to $V^\ast$, covered by a $\operatorname{Homeo}_0(X)$-bundle $(\lbar{E}; E, E^\ast)$, constant near $\de V$, such that $V^\ast$ possesses a $C^0$-transverse foliation $\mathcal{F}$ agreeing with $\mathcal{F}_0$ on some smaller neighborhood $\EuScript{N}_0(\de V) \subset \EuScript{N}_1(\de V)$.
\end{thm}

It is helpful to think of Theorem \ref{thm:extend1}, and the others, in the language of problem solving. The initial bundle $B$ over $V$ with foliation $\mathcal{F}_0$ near $\de V$ is the ``problem'' and the ``solution'' is the bundle at the opposite end $V^\ast$ of the cobordism, or semi-$s$-cobordism (ssc), with $\mathcal{F}_0$ extending from a germ of $\de V$ to $\mathcal{F}$ over all of $V^\ast$. We think of the poblem as \emph{flattening} an initial bundle while changing the base as little as possible and the boundary conditions $\mathcal{F}_0$ not at all. Given a problem on $V$, prob, and another manifold without boundary $Q$ we can pose a new stabilized problem, prob$\times Q$ which is the bundle
\begin{tikzpicture}[anchor=base, baseline]
    \node at (-0.1,1.1) {$X$};
    \draw [->] (0.1,1.2) -- (0.5,1.2);
    \node at (1.1,1.1) {$B \times Q$};
    \draw [->] (1.1,1) -- (1.1,0.4);
    \node at (1.1,0) {$V \times Q$};
\end{tikzpicture}
with foliation near $\de V \times Q = \de(V \times Q)$ defined as the pull back of $\mathcal{F}_0$ under $\mathrm{pr}_1: \de V \times Q \ra \de V$. Note that the original foliation and its pull back are both codimension $= q = \dim(X)$. When we speak of a solution to a stabilized problem there is no requirement that the solution has any product structure.

\begin{thm}\label{thm:extend2}
    Suppose $B$ has structure group lying within $\operatorname{bilipschitz}(X)$, but no assumption of lying within the identity component, and that there is a transverse foliation $\mathcal{F}_0$ of class at least bilipschitz over a neighborhood $\de V$. This is our problem, $\mathrm{prob}$. If $\dim(V) = p = 3$, then by Theorem \ref{thm:extend1} $\mathrm{prob}$ is ssc to a solution of bilipschitz class. If $\dim(V) = p \geq 4$ then for $Q = \prod_{j=1}^J S^{i_j}$, $i_j \geq 1$, $\mathrm{prob} \times Q$ is ssc to a solution of bilipschitz class. The number of factors $J$ is a function of the problem data and can be enormous, but the choice of the factor sphere dimensions $i_j \geq 1$ is completely arbitrary.
\end{thm}

\begin{note}
    When Theorem \ref{thm:extend2} is restricted $p = \dim(V) = 3$ there is a subtle difference with Theorem \ref{thm:extend1}: the structure group has changed from $\homeo_0(X)$ to bilipschitz, since the latter is the lowest differentiability where the methods of \cite{mei} apply, and these methods are not confined to the identity component. So the structure group is, simultaneously, more and less restrictive than before.
\end{note}

One may hope that the stabilization in Theorem \ref{thm:extend2} for $p \geq 4$ is unnecesary, as we know no counterexamples. The next theorem solves the initial problem without stabilization but with less control on the cobordism $W$.

\begin{thm}\label{thm:four}
    Suppose $\dim(V) \geq 3$ and $B$ has structure group $\homeo(X)$ (or bilipschitz$(X)$) and that there is a $C^0$-transverse foliation near $\de V$. Then there is a \emph{solution}, a cobordism $(W;V, V^\ast)$ covered by bundles to a solution
    \begin{tikzpicture}[anchor=base, baseline]
        \node at (0,1.1) {$X$};
        \draw [->] (0.2,1.2) -- (0.8,1.2);
        \node at (1.1,1.1) {$B^\ast$};
        \draw [->] (1.1,1) -- (1.1,0.4);
        \node at (1.1,0) {$V^\ast$};
    \end{tikzpicture}
    with a topologically flat connection, \emph{and} with $W$ admitting a retraction $r: W \ra V$, $r \circ \mathrm{inc}_V = \id_V$.
\end{thm}

\begin{note}
    A consequence of the retraction $r$ is the existence of the degree 1 map $r \circ \mathrm{inc}_{V^\ast}: V^\ast \ra V$. Degree one maps induce a well studied partial order on manifolds, e.g.\ the cohomology of the target always injects into the source, and $\pi_1(V^\ast) \ra \pi_1(V)$ is surjective. So as with a ssc, $V$ is the ``simple'' and $V^\ast$ the ``complex'' end. Also as with ssc, finite composition of retracting cobordisms are also retracting.
\end{note}

The next theorem provides, in the Lie context, a ssc from problem to solution, whenever the homological obstruction vanishes.\footnote{Regarding this obstruction, we thank the referee for pointing out that for $G$ either compact or complex semi-simple the map $H_\ast(BG^\delta;Q) \ra H_\ast(BG;Q)$ is trivial, see\cite{mil83}. So in these cases, the analogous rational obstruction automatically vanishes.}

\begin{thm}\label{thm:refine1}
    Let $G$ be a simply connected semi-simple Lie group with Lie algebra $\mathfrak{g}$. Let $(V,\Sigma)$ be a compact 3-manifold with boundary $\Sigma$, and $B$ a principal $G$-bundle with a fixed flat $\mathfrak{g}$-connection $A$ over $\Sigma$ which is trivial in $H_2(G^\delta,\Z)$ i.e.\ bounds a flat principal $G$-bundle over some $V^\pr$, $\de V^\pr = \Sigma$. Then there is a ssc $(W,V,V^\ast)$ and an extension $\lbar{B}$ of $B$ over $W$ with $B^\ast$ over $V^\ast$ possessing a flat $\mathfrak{g}$-connection extending $A$ on $\Sigma$.
\end{thm}

\begin{ext}
    Theorem \ref{thm:refine1} is a special case of what we actually will prove. It is not necessary that the bundle $B$ be a $G$-principal bundle with flat $\mathfrak{g}$-connection near the boundary. Instead all that is required is that the fiber $X$ be a manifold with a point set (not necessarily Riemannian) metric and the monodromy of the topologically flat connection lie in the identity component of the group of isometries $I_0(X) \subset \operatorname{bilipschitz}_0(X)$. By classical theorems of Montgomery and Zippin, $I_0(X)$ is a Lie group and that forms the bridge. Call the extension \textbf{Theorem \ref{thm:refine1}\textprime}.
\end{ext}

The last theorem (\textbf{\ref{thm:refine2}}) quantifies the Milnor-Wood inequality \cites{mil66,wood} for circle bundles over surfaces, by considering the trade-off between adding lots of genus to the base and keeping the transition functions close to $\mathrm{U}(1) \subset \operatorname{PSL}(2,\R)$, and keeping the genus lower and allowing transition function (like large boosts from special relativity) which are quite far from rotations. We can keep the transition function within the Lie group $\operatorname{PSL}(2,\R) \subset \homeo_0(S^1)$, the embedding via the usual representation on the hyperbolic disk model.

Consider the problem of imposing a $\operatorname{sl}(2,\R)$ flat connection on a circle bundle $B$ with Euler class $\chi(B)$ over a base surface $\Sigma_g$ of genus $g$. Let $g_{\chi(B)}(\epsilon)$ be the smallest genus so that there is a generating\footnote{All generating sets are assumed to be ``symmetric'' i.e.\ closed under inverse.} set $S$ for $\pi_1(\Sigma_g)$ so that $B$ admits a topologically flat connection with representation $\rho: \pi_1(\Sigma_g) \ra \operatorname{PSL}(2,\R)$ so that $\rho(S) \subset \EuScript{N}_\epsilon(\mathrm{U}(1)) \subset \operatorname{PSL}(2,\R)$, where the $\epsilon$-neighborhood is defined used the sup norm. Define $\lbar{g}_{\chi(B)}^\pr(\epsilon)$ similarly by replacing $\operatorname{PSL}(2,\R)$ with $\homeo_0(S^1)$. Clearly $\lbar{g}_{\chi(B)}^\pr(\epsilon) \leq g_{\chi(B)}(\epsilon)$. It seems reasonable in light of \cite{wood} to guess that they are actually equal.

\begin{thm}\label{thm:refine2}
    The function $g_{\chi(B)}(\epsilon)$, for $\epsilon > 0$, is defined into the natural numbers union 0, $\N^+$. For $\epsilon$ sufficiently small it obeys the upper bound:
    \[
        g_{\chi(B)} \leq \frac{2\pi \abs{\chi(B)}}{\epsilon^2 - O(\epsilon^3)}
    \]
    where the error term satisfies $\abs{\frac{O(\epsilon^3)}{\epsilon^3}} < \mathrm{const.}$, some $\mathrm{const.} > 0$.
\end{thm}

One might ask for a rather strong converse of the form:

For sufficiently small $\epsilon > 0$, $g_{\chi(B)} \geq \frac{c \abs{\chi(B)}}{\epsilon^2 + O(\epsilon^3)}$, for some $c > 0$, perhaps with $c = 2\pi$. But we have only been able to prove\footnote{The proof: The standard surface relator has length $4g$ and is a composition of $4g$ elements of $\operatorname{Homeo}^+(S^1)$, each represents to an element $e_i$ within $\epsilon$ of some rotation $r_i$, $1 \leq i \leq 4g$. Estimating rotation numbers as in \cite{wood}, the rotation number of the composition of the $e_i$ is within $4g\epsilon$ of the rotation number of the composition of $r_i$, which is trivial. The former rotation number is in the Euler class.} such a statement where the denominator is $\epsilon$ (not $\epsilon^2$), and $S$ is restricted to the standard generators for $\pi_1(\Sigma_g)$. So this question is open.

Conjecture \ref{conj:gens}, below, proposes a vast generalization of Theorem \ref{thm:refine2} to higher dimensions and with the additional feature that when $\dim(\text{base}) \geq 3$, the topology of the solution can be controlled up to ssc. In theorems \ref{thm:extend1} and \ref{thm:extend2} we have already seen that it is sometimes possible to ``hide'' the extensive genus inside a perfect group, suggesting the conjecture may be true.

Before stating broad (and optimistic) conjectures, we give the simplest example of what we would like to know, but don't.

\begin{ex}
    Recall the 't Hooft instanton, also known as a generalized Hopf fibration
    \begin{tikzpicture}[anchor=base, baseline]
        \node at (0,1.1) {$S^3$};
        \draw [->] (0.2,1.2) -- (0.8,1.2);
        \node at (1.1,1.1) {$S^7$};
        \draw [->] (1.1,1) -- (1.1,0.4);
        \node at (1.1,0) {$S^4$};
    \end{tikzpicture}
    with structure group $\operatorname{SU}(2)$. We ask if for every sup norm neighborhood $\EuScript{N}(\operatorname{SU}(2)) \subset \operatorname{bilipschitz}_0(S^3)$. There is a homology 4-sphere $H$ so that the pull back bundle under the degree one map $H \ra S^4$,
    \begin{tikzpicture}[anchor=base, baseline]
        \node at (0,1.1) {$S^3$};
        \draw [->] (0.2,1.2) -- (0.8,1.2);
        \node at (1.1,1.1) {$E$};
        \draw [->] (1.1,1) -- (1.1,0.4);
        \node at (1.1,0) {$H$};
    \end{tikzpicture}
    , admits a topologically flat connection given by $\rho: \pi_1(H) \ra \operatorname{bilipschitz}_0(S^3)$ so that for some generating set $S$ of $\pi_1(H)$, $\rho(S) \subset \EuScript{N}(\operatorname{SU}(2))$. This seems to require a new idea, even without the condition on the generating set.
\end{ex}

\begin{conj}
    Theorem \ref{thm:extend2} holds without stabilization.
\end{conj}

\begin{conj}\label{conj:gens}
    Let $(V, \de V)$ be a manifold of $\dim(V) \geq 3$. Let
    \begin{tikzpicture}[anchor=base, baseline]
        \node at (0,1.1) {$X$};
        \draw [->] (0.2,1.2) -- (0.8,1.2);
        \node at (1.1,1.1) {$B$};
        \draw [->] (1.1,1) -- (1.1,0.4);
        \node at (1.1,0) {$V$};
    \end{tikzpicture}
    be a bundle with structure group $\homeo(X)$, where $X$ is any manifold with a (point set) metric, and the bundle possessing near $\de V$ a topologically flat connection $\mathcal{F}_0$ with holonomy lying in $I(X)$, the group of isometries of the fiber, and let $\EuScript{N}$ be any norm-topology neighborhood of $I(X)$ in $\homeo(X)$. Then there exists a ssc $(W; V,V^\ast)$, constant near $\de V$, covered by a bundle $\lbar{B}$ with structure group $\homeo_0(X)$ to a bundle
    \begin{tikzpicture}[anchor=base, baseline]
        \node at (1.1,1.1) {$B^\ast$};
        \draw [->] (1,1) -- (1,0.4);
        \node at (1.1,0) {$V^\ast$};
    \end{tikzpicture}
    possessing a topologically flat connection inducing a representation $\rho: \pi_1(V^\ast) \ra \homeo(X)$, with the property that $\rho(S) \subset \EuScript{N}$ for some generating set $S$ for $\pi_1(V^\ast)$.
\end{conj}

\section{Dynamics}
We use a somewhat more controlled version of the Fisher-Epstein Theorem \cites{fish,eps70}.

\begin{prop}
    The identity component of the homeomorphism group of every manifold is simple.
\end{prop}

Fisher actually worked under a stability assumption which became redundant in 1969 when Kirby \cite{kirby} proved all homeomorphisms are stable as part of his work on the Annulus Conjecture.\footnote{The 4-dimensional case was completed by Quinn \cite{quinn} in 1982.} While adding some control we rely heavily on Fisher's identities. Again all constructions may be done in the bilipschitz category.

Remarkably, it is possible to write the general $f \in \homeo_0(X)$ as a product of conjugates of any single homeomorphism $h$ and its inverse $h^{-1}$, where $h \neq \id \in \homeo_0(X)$ is arbitrary. By choosing the essential support $\operatorname{supp}(h)$ near a fine net $N \subset X$ it is further possible to control the support of all conjugators and in particular keep them small in the norm topology. Let us take $N_\epsilon$ to be a maximal collection of points with all distances $\geq \frac{\epsilon}{2} > 0$ and $\{\operatorname{U}_i\}$ the open cover of $\epsilon$-balls about $N_\epsilon$. Here is the construction using the concept of \emph{fragmentation} (see \cite{nariman}).

Since $\homeo_0(X)$ is perfect and has the fragmentation property w.r.t.\ the open cover $\{U_i\}$ we may write $f = \prod_{i=1}^r [a_i,b_i]$, $a_i,b_i$ supported in $\mathrm{U}_i$. It suffices to write each $[a_i,b_i]$ as product of conjugates of $h$ and $h^{-1}$. By assumption there is a $\delta$-neighborhood $V$ of $N$, $\delta << \frac{\epsilon}{2}$, so that $V \cap h(V) = \varnothing$. Let $g_i$ be a radial compression (a homeomorphism) supported near $\operatorname{U}_i$ so that $g_i(\operatorname{U}_i) \subset V_i$, $V_i$ being the component of $V$ containing the $i$th element of $N$. Writing $x^y$ for $yxy^{-1}$, and $\ubar{a}_i$ for $a_i^{g_i}$ and $\ubar{b}_i$ for $b_i^{g_i}$, similar to \cite{tsu08} (see Remark 6.6), we may write an identity expressing the basic commutators $[\ubar{a}, \ubar{b}]$ as a product of two conjugates of $h$ and two conjugates of $h^{-1}$ for a total of four conjugates. For convenience set $\ubar{c}_i = h^{-1} \ubar{a}_i h$, and on line (\ref{eq:ab_conj}) we suppress the subscript $i$ for readability to obtain
\begin{equation}
\begin{split}\label{eq:ab_conj}
    [\ubar{a},\ubar{b}] & = \ubar{a}\ubar{b}\ubar{a}^{-1}\ubar{b}^{-1} = h(h^{-1}\ubar{a}h)h^{-1}(\ubar{b}\ubar{a}^{-1}\ubar{b}^{-1}) \\
    & = h(\ubar{c}h^{-1}\ubar{c}^{-1})(\ubar{c}\ubar{b}h\ubar{c}^{-1}\ubar{b}^{-1})(\ubar{b}h^{-1}\ubar{b}^{-1}) \\
    & = h(h^{-1})^{\ubar{c}}(h^{\ubar{c}\ubar{b}})(h^{-1})^{\ubar{b}}
\end{split}
\end{equation}
where in the final equality we have used that $\ubar{c}$ and $\ubar{b}$ commute owing to there having disjoint support. So, restoring the $i$ subscript
\begin{equation}\label{eq:ab_conj_i}
    [\ubar{a}_i,\ubar{b}_i] = g_i^{-1}[\ubar{a}_i,\ubar{b}_i]g_i = h^{g_i^{-1}}(h^{-1})^{g_i^{-1}h^{-1}g_ia_ig_i^{-1}hg_i}h^{g_i^{-1}h^{-1}g_ia_ig_i^{-1}hg_ib}(h^{-1})^{g_ib_ig_i^{-1}}
\end{equation}

This yields
\begin{prop}\label{prop:x_conjugators}
    Let $X$ be a (metric) topological manifold and $f \in \homeo_0(X)$, then $f$ may be written as a product of conjugates of one homeomorphism $h$ and its inverse. If $h$ has essential support $\epsilon$-close to a net $N \subset X$, the conjugators, displayed in line (\ref{eq:ab_conj_i}) are themselves $O(\epsilon)$ in sup-norm. If $X$ is a bilipschitz manifold then any $f \in \operatorname{bilipschitz}_0(X)$ can be similarly written with all letters now bilipschitz homeomorphisms, with the parallel assertion on sup-norm still holding.
    \qed
\end{prop}

The ``smallness'' conclusion of Proposition \ref{prop:x_conjugators} are not required for the theorems proven in this paper, but are stated as they should be useful in studying Conjecture \ref{conj:gens} and the instanton example.

\begin{ext}
    Each $a_i$ and $b_i$ above can themselves be factored as on line (\ref{eq:ab_conj_i}) (being elements of $\homeo_0(X)$), thus $f$ may further be written as a product of commutators of elements which are themselves conjugates of the chosen $h$. This will be exploited in section 5.
\end{ext}

\section{The Local Models}
We will use explicit models $C_1$, $C_2$, $\lbar{C}_1$, $\lbar{C}_2$, and more general models, written in script and denoting a composition of the corresponding capital letter model $\mathcal{C}_1$, $\mathcal{C}_2$, $\lbar{\mathcal{C}}_1$, and $\lbar{\mathcal{C}}_2$. Both $C_1$ and $C_2$ are homology solid tori and $\lbar{C}_1$ and $\lbar{C}_2$ are fixed cobordisms, rel.\ boundary, to the standard solid torus $S^1 \times D^2$, so $\lbar{C}_1$ and $\lbar{C}_2$ are homology $S^1 \times D^3$'s. Both $C_1$ and $C_2$ will come with a homomorphism $\rho_i: \pi_1(C_i) \ra$ (a relevant structure group) for $i = 1,2$. $C_i$ is the closed complement of certain knot $k_i$ in a homology 3-sphere $\Sigma_i$, and the corresponding homology cobordisms $\lbar{C}_i$ will be the closed slice complements for $k_i$ in a corresponding homology 4-ball, $\mathcal{B}_i$. After two short paragraphs on preliminaries, precise definitions are given.

Let $P_1$ and $P_2$ be the Siefert fibered manifolds over $S^2$ with $(2,3,7)$ and $(2,3,5)$ multiplicities, respectively, describing the (three) exceptional fibers. $P_1$ is an $\tld{\operatorname{SL}}(2;\R)$-manifold whereas $P_2$ is spherical (actually it is the Poincare homology sphere). Let $P^-_i$, $i = 1,2$, denote the punctured manifold. Let $(\gamma_1, \de \gamma_1)$ be any essential embedded arc in $(P_1^-, \de)$ and let $(\gamma_2, \de \gamma_2)$ be any embedded arc in $(P_2^-, \de)$ representing an element of the binary icosahedral group $BI \cong \pi_1(P_2)$ of order $>2$. Note that $\pi_1(P_1)$ is torsion free, being an infinite cyclic extension of a hyperbolic triangle group.

For $i = 1,2$, let $\Sigma_i = P_i \# -P_i$ be the connected sum of the homology sphere and its mirror image, $\Sigma_i = P^-_i \bigcup_{S^2} -P^-_i$, where $P^-_i$ is the punctured homology sphere. The knot $k_i \coloneqq \gamma_i \cup -\gamma_i \subset P^-_i \bigcup_{S^2} -P^-_i$ is slice in the homology 4-ball $\mathcal{B}_i \cong P^-_i \times I$, $\de \mathcal{B}_i = \Sigma_i$, with the (slice disk)$= \gamma_i \times I$.

Let $C_i$ denote the closed complement $\Sigma_i \setminus \EuScript{N}(k_i)$, an integral homology $S^1 \times D^2$, and $\lbar{C}_i$ the closed complement ($\mathcal{B}_i\setminus$slice disk). We may view $\lbar{C}_i$ as a relative $H_\ast$-cobordism from $C_i$ to $S^1 \times D^2$. $\lbar{C}$ will be useful in building certain cobordisms $W_0$. (We use the subscript 0 to distinguish these model cobordisms from the general cobordism denoted throughout by $W$.)

There are natural maps:
\begin{equation}\label{eq:natmaps}
    [m],[l] \in \pi_1\tikzmark{a}C_i \xrightarrow{\mathrm{inc}_{\#}} \pi_1 P_i \ast \pi_1 P_i \xrightarrow{\mathrm{proj}_1} \pi_1 \tikzmark{b}P_i
    \tikz[overlay,remember picture] {\draw[->,square arrow] (a.south) to (b.south);}
    \tikz[overlay,remember picture] {\node at (-2.4,-0.6) {\tiny{$\alpha$}};} \\[1em]
\end{equation}
with $\alpha[l]$ having order $>2$ and $\alpha[m]$ trivial, $m$ the meridian of $k$ and $l$ the longitude.

$\pi_1(P_i)$ are central extensions of the $(2,3,7)$ and $(2,3,5)$ triangle groups respectively. For $i=1$ $\pi_1(P_i)$ is torsion free so the $\alpha[l]$ is merely required to be nontrivial. For $i = 2$ the aim is to exclude the central element of $\pi_1(P_2) \cong BI$, the binary icosahedral group. This is easily done by avoiding that element when choosing $\gamma_2$.

Given a homology sphere $P$ there is a ``spun'' homology sphere $Q$ in one higher dimension with $\pi_1(P) \cong \pi_1(Q)$. $Q$ is an open book with fiber $P^-$ and identity monodromy:
\begin{equation}
    Q = P^- \times I \slash\hphantom{.}^{p \times 0 \equiv p \times 1,\ p \in P^-}_{q \times s = q \times s^\pr,\ q \in \de P^-,\ s,s^\pr \in I = [0,1]}
\end{equation}

Consequently the preceeding set of examples can be extended to all dimensions:

Let $P_i^k$, $k \geq 3$, be $P_i$ spun to dim $k$, $i = 1,2$, so $P_i^3 \coloneqq P_i$. $\Sigma_i^k$, $\mathcal{B}_i^k$ are the corresponding doubled homology spheres and homology ball and $\gamma_i \times I$ is still a 2D slice disk in $\mathcal{B}_i^k$, with boundary $k_i^k$. To explain this notation, $k_i$ is a one-dimensional loop, a knot, and its superscript (later dropped) indicates the dimension of the homology ball on whose boundary it lies. Call the closed complements in $\Sigma_i^k$ and $\mathcal{B}_i^k$, $C_i^k$ and $\lbar{C}_i^k$ respectively. Thus $C_i^k \coloneqq \Sigma_i^k \setminus \EuScript{N}(k_i)$, and $\lbar{C}_i^k \coloneqq \mathcal{B}_i^k \setminus \EuScript{N}(\gamma_i \times I)$. $C_i^k$ is an integral homology $S^{k-2} \times D^2$ and $\lbar{C}_i^k$ is a homology $S^{k-2} \times D^3$ which may be viewed as a $\Z$-homology cobordism, constant over the boundary, to the standard $S^{k-2} \times D^2$. Line \ref{eq:natmaps} continues to hold except that when $k > 3$, there is no \emph{meridional} class in $\pi_1(C_i^k)$ since the meridinal boundary factor is $S^{k-2}$, which is simply connected. This explains the capital models. Next we discuss their composition, denoted by the corresponding script letters. Having introduced the superscript $k$ to keep track of the dimension of the model, we will generally omit it in what follows when it is clear from context, for example from the dimension of the manifold being modified.

We will use two basic operations to build $\mathcal{C}$ from $C$ and $\lbar{\mathcal{C}}$ from $\lbar{C}$. The first is (iterated) longitudinal boundary connected sum, which we call longitudinal sum, for the homology solid tori (in any dimension $C \coloneqq (S^{k-2} \times D^2)_H$) with boundary $\cong S^{k-2} \times S^1$. (When the dimension $k = 3$ a normal framing specifies the ``longitude'' and fixes the identification of $\de(S^1 \times D^2)_H$ with $S^1 \times S^1$.) In general, the sphere $S^{k-2}$-factor is the longitude and if $I \subset S^1$ is a fixed interval the longitudinal sum is a gluing
\[
    C\ \nsharp_{\text{long}} C^\pr = C \independent C^\pr \slash S^{k-2} \times I \subset \de C \text{ identified via } (\id_{S^{k-2}} \times \te) \text{ to } S^{k-2} \times I \subset \de C^\pr
\]
$\te$ the reflection on $I$. Similarly:
\[
    \lbar{C}\ \nsharp_{\text{long}}\lbar{C}^\pr = \lbar{C} \independent \lbar{C}^\pr \slash S^{k-2} \times I \times I \subset \de \lbar{C} \text{ identified via } (\id_{S^{k-2}} \times \te \times \id_I) \text{ to } S^{k-2} \times I \times I \subset \de \lbar{C}^\pr
\]
where the last interval factor is normal to $C$ in $\lbar{C}$. This is illustrated in Fig.\ \ref{fig:connsum} for $k=3$ in the case of 4-fold iterations.

The second way in which we compose homology solid tori is ``Bing doubling,'' implanting a pair of $C$'s in a standard $S^{k-2} \times D^2$ along a Bing double of its core, or in dimension $k > 3$ spun Bing doubles (see \cite{krush} and our Figure \ref{fig:bingd}). The Bing doubling operation relates easily to group commutators, whereas the first operation relates to group multiplication. We actually require \emph{ramified} Bing doublings \cite{fq}, a simple extension corresponding to \emph{products} of commutators, explained below. In the case that bars are present, the Bing doubling transformation $\lbar{C} \ra \lbar{\mathcal{C}}$, is not \emph{replacement}, but the \emph{attachment} of two homology cobordisms $(\lbar{C}, C)$ along Bing pairs. On the boundary, this attachment agrees with the former operation of \emph{replacement}.

We change capital letters to script to indicate that our models have been composed according to longitudinal sums and Bing doubling.

For us the use of these models, particularly the un-barred $\mathcal{C}_i$ will be to fill in what Thurston \cite{thur76} calls \emph{holes} in foliations, which later morphed into the theory of fissures \cite{mei}.

\section{Proofs}
For the proof of Theorem \ref{thm:extend1}, the MT Theorem \ref{thm:mt} may be treated as a black box; we just need the statement that $W$ exists and then use surgery techniques and the models $C_1$ and $\lbar{C}_1$ to improve $W$ to a ssc while retaining the flat connection over its right end.

In contrast, the proof of Theorem \ref{thm:extend2} requires following in some detail Meigniez's proof of MT \cite{mei} and intervening at the correct moment with models $C_1^l$ and $\lbar{C}_1^l$, for various dimension $l$.

\begin{proof}[Proof of Theorem \ref{thm:extend1}]
    Recall the output of the MT Theorem is a cobordism of $X$-bundles:
    \begin{figure}[!ht]
        \centering
        \begin{tikzpicture}
            \node at (0,0) {$B\ \hookrightarrow \lbar{B}\ \hookleftarrow\ B^\ast$};
            \node at (0.1,-1) {$V \hookrightarrow W^4 \hookleftarrow V^\ast$};
            \draw[->] (-0.1,-0.3) -- (-0.1,-0.7);
            \draw[->] (1.1,-0.3) -- (1.1,-0.7);
            \draw[->] (-1.1,-0.3) -- (-1.1,-0.7);
        \end{tikzpicture}
    \end{figure}

    Give $W^4$ a smooth handle decomposition relative to $V$ and cancel any 0-handles and 4-handles (without changing $W$). Since $V$ (and $V^\ast$) may be presumed connected, any 1-handle (and any 3-handle) may be traded for a 2-handle at the expense of modifying $W$ by a trivial 1-surgery. It is easy to modify interior $W$ in this way but we need to carry the bundle cobordism $\lbar{B}$ along. This is one of two places the proof will use that the structure group of $B$ and hence $\lbar{B}$ is $\homeo_0(X)$. This guarantees that $\lbar{B}\vert_\gamma$ is trivial for any simple closed curve $\gamma$ (scc) in $W$, allowing the 1-surgery to be covered by a relative cobordism of bundles $(\lbar{B};B;B^\pr)$ over a (5D) cobordism of $W$. Three handles are dealt with as 1-handles by turning $W$ upside down. At this point we have reduced to the case where $W$ has handles of index 2 only. The attaching regions of these 2-handles determine a framed link $L \subset V$ and dually the 2-handle co-cores determine a framed link $L^\ast \subset V^\ast$. Framed surgery along $L$ produces $V^\ast$ from $V$, and dually framed surgery along $L^\ast$ produces $V$ from $V^\ast$. If we could actually do surgery on $L^\ast$ and propogate the topologically flat connection across the surgery we would have solved the flattening problem without changing $V$ by a cobordism. This cannot be done, generally, because surgery bounds each framed longitude of $L^\ast$ by a disk $D$ which can only be covered by a flat bundle if the holonomy around the longitude, $\de D$, is trivial. Traditionally, this problem is approached by replacing $D$ with a genus $g$ surface with circle boundary $S_g$ and exploiting the commutator structure of the boundary. But replacing a disk $D$ with a surface $S_g$ adds homology and is a much coarser modification than replacement by a ssc manifold. While we do need the commutator structure of $\pi_1(S_g)$, in dimension $\geq 3$, we are able to \emph{hide} it in a ssc manifold. So what we do instead is a homologically more subtle replacement: There is a homological version of surgery on $L^\ast$ employing the models $C_1$ (and $\lbar{C}_1$) which \emph{does} propogate the flat connection to $V^\ast$ and changes, in the end, $V$ only slightly, by a ssc. See Fig.\ \ref{fig:sscholo} to visualize these manipulations, and for a pictorial summary of the proof plan. The idea that the surface group relation can \emph{hide} inside a 3-manifold without producing first homology is as familiar as an incompressible surface in a homology 3-sphere. Glancing ahead to Figure \ref{fig:rambingd} and line (\ref{eq:abj}) we see the general surface relator but the homology classes $a_i$ and $b_i$ vanish in the model $\mathcal{C}_1$.

    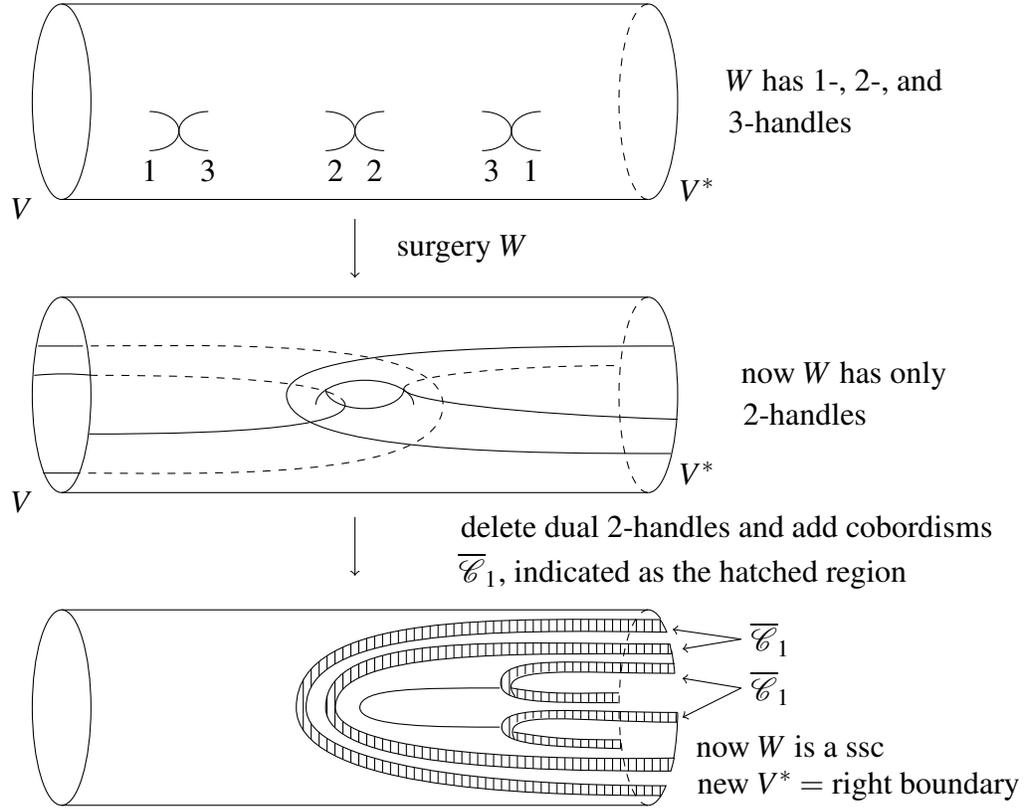
\begin{figure}[!ht]
        \centering
        \begin{tikzpicture}[scale=1.3]
            \draw (-2.95,2) ellipse (0.3 and 1);
            \draw (3.05,3) arc (90:-90:0.3 and 1);
            \draw[dashed] (3.05,3) arc (90:270:0.3 and 1);
            \draw (-2.95,3) -- (3.05,3);
            \draw (-2.95,1) -- (3.05,1);
            \draw (-2.05,1.9) arc (90:-90:0.3 and 0.2);
            \draw (-1.45,1.9) arc (90:270:0.3 and 0.2);
            \draw (-0.25,1.9) arc (90:-90:0.3 and 0.2);
            \draw (0.35,1.9) arc (90:270:0.3 and 0.2);
            \draw (1.35,1.9) arc (90:-90:0.3 and 0.2);
            \draw (1.95,1.9) arc (90:270:0.3 and 0.2);
            \node at (-2.05,1.3) {1};
            \node at (-1.45,1.3) {3};
            \node at (-0.15,1.3) {2};
            \node at (0.25,1.3) {2};
            \node at (1.45,1.3) {3};
            \node at (1.85,1.3) {1};
            \node at (-3.35,0.9) {$V$};
            \node at (3.55,1.1) {$V^\ast$};
            \node at (4.95,2.2) {$W$ has 1-, 2-, and};
            \node at (4.5,1.8) {3-handles};
            \draw[->] (0.05,0.8) -- (0.05,0.2);
            \node at (1.15,0.5) {surgery $W$};

            \draw (-2.95,-1) ellipse (0.3 and 1);
            \draw (3.05,0) arc (90:-90:0.3 and 1);
            \draw[dashed] (3.05,0) arc (90:270:0.3 and 1);
            \draw (-2.95,0) -- (3.05,0);
            \draw (-2.95,-2) -- (3.05,-2);
            \node at (-3.35,-2.1) {$V$};
            \node at (3.55,-1.8) {$V^\ast$};
            \draw(-3.14,-1.8) -- (-2.84,-1.8);
            \draw[dashed] (-2.84,-1.8) .. controls (-1.55,-1.8) and (0.95,-1.9) .. (0.95,-1.1) .. controls (0.95,-0.4) and (-1.45,-0.5) .. (-2.81,-0.5);
            \draw(-2.81,-0.5) -- (-3.19,-0.5);
            \draw[dashed] (0.55,-0.94) .. controls (0.55,-0.8) and (1.65,-0.7) .. (2.75,-0.7);
            \draw (0.55,-0.94) .. controls (0.55,-1.05) and (1.65,-1.2) .. (3.35,-1.25);
            \draw (-0.05,-1.1) .. controls (-0.05,-1.4) and (-2.05,-1.4) .. (-2.67,-1.4);
            \draw[dashed] (-0.05,-1.1) .. controls (-0.1,-0.8) and (-2.05,-0.8) .. (-2.63,-0.8);
            \draw (-2.63,-0.8) to[out=175,in=5] (-3.23,-0.8);
            \draw (3.3,-0.5) .. controls (1.85,-0.5) and (-0.65,-0.5) .. (-0.65,-1) .. controls (-0.65,-1.6) and (1.85,-1.6) .. (3.28,-1.6);
            \draw (0.65,-1.1) arc (0:180:0.5 and 0.25);
            \draw (0.55,-0.94) arc (0:-180:0.4 and 0.2);
            \node at (5.05,-0.8) {now $W$ has only};
            \node at (4.65,-1.2) {2-handles};
            
            \draw[->] (0.05,-2.25) -- (0.05,-2.85);
            \node at (3.85,-2.35) {delete dual 2-handles and add cobordisms};
            \node at (3.4,-2.8) {$\lbar{\mathcal{C}}_1$, indicated as the hatched region};

            \draw (-2.95,-4.2) ellipse (0.3 and 1);
            \draw[dashed] (3.05,-3.2) arc (90:270:0.3 and 1);
            \draw (3.05,-3.2) arc (90:65:0.3 and 1);
            \draw (3.05,-5.2) arc (270:295:0.3 and 1);
            \draw (-2.95,-3.2) -- (3.05,-3.2);
            \draw (-2.95,-5.2) -- (3.05,-5.2);
            \draw[pattern = vertical lines] (3.15,-5.1) .. controls (1.65,-5.1) and (-0.55,-5) .. (-0.55,-4.2) .. controls (-0.55,-3.2) and (1.65,-3.3) .. (3.17,-3.3) arc (68:53:0.3 and 1) .. controls (1.45,-3.4) and (-0.45,-3.4) .. (-0.45,-4.2) .. controls (-0.45,-4.8) and (1.45,-5) .. (3.23,-5) arc (-53:-65:0.3 and 1) -- cycle;
            \draw[pattern = vertical lines] (3.27,-3.55) .. controls (1.64,-3.55) and (-0.25,-3.4) .. (-0.25,-4.2) .. controls (-0.25,-4.8) and (1.75,-4.85) .. (3.28,-4.85) arc (-45:-35:0.3 and 1) .. controls (1.45,-4.75) and (-0.15,-4.6) .. (-0.15,-4.2) .. controls (-0.15,-3.6) and (1.74,-3.65) .. (3.29,-3.65) arc (33:38:0.3 and 1) -- cycle;
            \draw (3.32,-3.75) .. controls (2.34,-3.75) and (1.54,-3.65) .. (1.54,-3.95) .. controls (1.54,-4.15) and (1.94,-4.15) .. (2.74,-4.15);
            \draw (3.32,-3.85) .. controls (2.35,-3.85) and (1.43,-3.75) .. (1.69,-4.09);
            \draw (1.65,-4) to [out=-10,in=180] (2.75,-4.05);
            \draw (3.32,-3.85) arc (18:33:0.3 and 1);
            \fill[pattern = vertical lines] (3.32,-3.75) .. controls (2.34,-3.75) and (1.54,-3.65) .. (1.54,-3.95) .. controls (1.54,-4.15) and (1.94,-4.15) .. (2.74,-4.15) -- (2.75,-4.05) to[out=180,in=-10] (1.65,-4) .. controls (1.58,-3.75) and (2.35,-3.85) .. (3.32,-3.85) -- cycle;
            \draw (3.35,-4.25) .. controls (2.35,-4.25) and (1.55,-4.15) .. (1.55,-4.4) .. controls (1.55,-4.6) and (2.15,-4.6) .. (2.77,-4.62);
            \draw (3.35,-4.35) .. controls (2.35,-4.35) and (1.52,-4.2) .. (1.68,-4.52);
            \draw (1.67,-4.45) to[out=-10,in=180] (2.77,-4.53);
            \draw (3.35,-4.25) arc (-2:-30:0.3 and 1);
            \fill[pattern = vertical lines] (3.35,-4.25) .. controls (2.35,-4.25) and (1.55,-4.15) .. (1.55,-4.4) .. controls (1.55,-4.6) and (2.15,-4.6) .. (2.77,-4.62) -- (2.77,-4.53) to [out=180,in=-10] (1.67,-4.45) .. controls (1.6,-4.2) and (2.35,-4.35) .. (3.35,-4.35) -- cycle;
            \draw (1.53,-4) .. controls (0.7,-4) and (0.1,-4) .. (0.1,-4.2) .. controls (0.1,-4.4) and (0.9,-4.4) .. (1.53,-4.4);
            \node at (4.3,-3.5) {$\lbar{\mathcal{C}}_1$};
            \draw[->] (4,-3.5) -- (3.3,-3.4);
            \draw[->] (4,-3.5) -- (3.4,-3.6);
            \node at (4.3,-4) {$\lbar{\mathcal{C}}_1$};
            \draw[->] (4,-4) -- (3.4,-3.9);
            \draw[->] (4,-4) -- (3.4,-4.3);
            \node at (4.5,-4.6) {now $W$ is a ssc};
            \node at (5.2,-5) {new $V^\ast =$ right boundary};

            \node at (-5.5,0) {\hspace{0.5em}};
        \end{tikzpicture}
        \caption{$\lbar{\mathcal{C}}_1$ are composition built from several copies of $\lbar{C}_1$}\label{fig:sscholo}
    \end{figure}

    Next we explain how the models $\lbar{\mathcal{C}}_1$ are constructed to permit an extension of the topologically flat connection (i.e.\ the representation to $\homeo_0(X)$), and why the cobordisms indicated in Figure \ref{fig:sscholo} and the discussion above, is in fact a ssc.

    To apply Proposition \ref{prop:x_conjugators}, we need a homomorphism from $\pi_1(C_1)$ to $\homeo_0(X)$, factoring through a homomorphism $\rho: \pi_1(P_1) \ra \homeo_0(X)$ so that for some element $\alpha \in \pi_1(P_1)$, $\rho(\alpha)$ has the property of $h$ in Proposition \ref{prop:x_conjugators}, that for a fine net $\lgast$, $\rho(\alpha) \lgast \cap \lgast = \varnothing$. Actually we now construct a representation $\rho$ so that for every $\alpha \neq e \in \pi_1(P_1)$, $\rho(\alpha)$ has this property.

    As Thurston \cite{thur74c} observed, $\pi_1(P_1)$ is naturally a subgroup of $\tld{\operatorname{PSL}}(2,\R)$, the universal cover of $\operatorname{PSL}(2,\R)$. $\operatorname{PSL}(2,\R)$ is the group of oriented isometries of the hyperbolic plane $H^2$ and acts faithfully on the circle at infinity $S^1$ via M\"{o}bious transformations. Thus $\tld{\operatorname{PSL}}(2,\R)$ acts smoothly on the real line $\R = \tld{S}^1$, and taking the end compactification we obtain a continuous (and bi-Lipschitz) action of $\tld{\operatorname{PSL}}(2,\R)$ on the closed interval $I \coloneqq [0,1]$. This action is faithful. It is easy to promote this action to a representation $\pi_1(P_1) \ra \homeo(D^n; \id \text{ on } \de)$. To do so implant this action on an interval $I \subset S^1$, then suspend and delete a fixed disk; one obtains a faithful action on $D^2$. Iterating implantation, suspension, and deletion, one obtains a faithful action on the $n$-cell $D^n$. Finally, given the ``net'' $\lgast \subset X$, we produce $h$ as in section 3 by allowing $\rho$ to act on a small ball around each point $\ast \in \lgast$. We employ such an $h$ to make our construction as local as possible, with an eye toward future applications. Then for any $\alpha \neq e \in \pi_1(P_1)$ the composition $\rho(\alpha)$
    \begin{equation}
        \rho: \pi_1(P_1) \ra \tld{\operatorname{PSL}}(2,\R) \ra \homeo_0(X),\ \alpha \mapsto \rho(\alpha)
    \end{equation}
    has the property of $h$ in Proposition \ref{prop:x_conjugators}. Note that since $\pi_1(P_1)$ is perfect, Thurston's stability theorem says that there is no nontrivial representation $\pi_1(P_1) \ra \operatorname{Diff}^1(D^n, \de D^n)$, so this bilipschitz category construction cannot be carried out differentiably.

    As illustrated in Figure \ref{fig:sscholo}, we construct the ssc $W$ as $V \times I \cup W_0$, where $W_0$, built from a collection of $\lbar{C}_1$ ($W_0$ is a disjoint union of $\lbar{\mathcal{C}}_1$), is itself a ssc on $N \cong \coprod S^1 \times D^2$, a neighborhood of $L^\ast$,
    \tikz[anchor=base, baseline] {\node at (0,0) {$N \hookrightarrow \lbar{\mathcal{C}}_1 \hookleftarrow \mathcal{C}_1$};
    \node at (-0.7,0.2) {$\xleftarrow{r}$};}
    (constant on the boundary), so that the new $V^\ast$ has the form $V^\ast = (V \setminus N) \cup (\independent \lbar{\mathcal{C}}_1)$ and has a topological flat connection agreeing with the restrictions to the flat connection $A$ over $\de(V)$.

    The connection $A$ is trivial on a meridian to $L^\ast$ but in general is nontrivial on the framed longitude $l_i$. Conventional surgery would bound each $l_i$ by a disk over which the connection cannot extend flatly (unless the holonomy around $l_i$ is trivial). This problem is solved by the homology surgery, using $\lbar{\mathcal{C}}$, which we now commence building.

    Since our model cobordism $(\lbar{C}_i; C_i, S^1 \times D^2)$ is a homology product with $\Z$-coefficients, but not group ring $\Z[\Z]$ coefficients, we next use a trick, Bing doubling, to isolate each copy of $\lbar{C}$ from the fundamental group, as illustrated in Figure \ref{fig:bingd}. The two components of the Bing double are null homotopic in the ambient solid torus so they and whatever is glued to them lift to the $\pi_1(W)$-cover: they do not unwrap. This allows us to easily compute homology in that cover. By contrast, when knot complements are unwrapped in a cover additional homology generally appears. For example the 6-fold cyclic cover of the trefoil knot complement is $S^1 \times T^2_-$, the circle cross a punctured torus.

    Surgery will glue in models schematically as in Fig.\ \ref{fig:bingd} but with additional ramifications\footnote{Ramification (see \cite{fq}) is a technical term for enhancing one Bing pair, as drawn inside the larger solid torus in Figure \ref{fig:bingd}, to $k$ such pairs. Each pair is successively closer to the boundary; look ahead to Figure \ref{fig:rambingd}. Note that this differs from taking parallel copies of the components of an unramified Bing double. Ramification corresponds to increasing the genus of the Siefert surface for the loop $\hat{m}_i$ of Figures \ref{fig:bingd} and \ref{fig:rambingd}.}:

    \begin{figure}[!ht]
        \centering
        \begin{tikzpicture}[scale=1.4]
            \draw (0,0) ellipse (3 and 1.5);
            \draw (0.5,0) arc (0:-180:0.5 and 0.25);
            \draw (0.45,-0.12) arc (0:180:0.45 and 0.15);
            \draw (0,1.5) arc (90:270:0.25 and 0.73);
            \draw[color=white, line width = 1mm] (-1.1,-0.1) to [out=70,in=110] (1.1,-0.1);
            \draw (-1.1,-0.1) to [out=70,in=110] (1.1,-0.1);
            \draw[color=white, line width = 1mm] (-1.3,-0.1) to [out=70,in=180] (0,0.67) to [out=0,in=110] (1.3,-0.1);
            \draw (-1.25,-0.05) to [out=70,in=180] (0,0.67) to [out=0,in=110] (1.25,-0.05);
            \draw (-1.8,-0.4) .. controls (-1.6,0.2) and (-0.8,-0.4) .. (0,-0.4) .. controls (0.8,-0.4) and (1.6,0.2) .. (1.8,-0.4);
            \draw (-1.7,-0.5) .. controls (-1.5,0) and (-0.8,-0.55) .. (0,-0.55) .. controls (0.8,-0.55) and (1.5,0) .. (1.7,-0.5);
            \draw[color=white, line width = 1mm] (-1.35,-0.35) .. controls (-1.6,-0.8) and (-2.2,-0.5) .. (-2,0.2) .. controls (-1.8,0.9) and (-0.5,1.1) .. (0,1.1) .. controls (0.5,1.1) and (1.8,0.9) .. (2,0.2) .. controls (2.2,-0.5) and (1.6,-0.8) .. (1.35,-0.35);
            \draw (-1.35,-0.35) .. controls (-1.6,-0.8) and (-2.2,-0.5) .. (-2,0.2) .. controls (-1.8,0.9) and (-0.5,1.1) .. (0,1.1) .. controls (0.5,1.1) and (1.8,0.9) .. (2,0.2) .. controls (2.2,-0.5) and (1.6,-0.8) .. (1.35,-0.35);
            \draw[color=white, line width = 1mm] (-1.2,-0.38) .. controls (-1.6,-1.1) and (-2.35,-0.5) .. (-2.17,0.15) .. controls (-2,1.1) and (-0.5,1.25) .. (0,1.25) .. controls (0.5,1.25) and (2,1.1) .. (2.17,0.15) .. controls (2.35,-0.5) and (1.6,-1.1) .. (1.2,-0.38);
            \draw (-1.2,-0.38) .. controls (-1.6,-1.1) and (-2.35,-0.5) .. (-2.17,0.15) .. controls (-2,1.1) and (-0.5,1.25) .. (0,1.25) .. controls (0.5,1.25) and (2,1.1) .. (2.17,0.15) .. controls (2.35,-0.5) and (1.6,-1.1) .. (1.2,-0.38);
            \draw (-1.9,-0.7) .. controls (-2.1,-1.3) and (-0.5,-1.4) .. (0,-1.4) .. controls (0.5,-1.4) and (2.1,-1.3) .. (1.9,-0.7);
            \draw (-1.73,-0.79) .. controls (-1.8,-1.1) and (-0.4,-1.25) .. (0,-1.25) .. controls (0.4,-1.25) and (1.8,-1.1) .. (1.73,-0.79);
            \draw (-0.8,0.39) ellipse (0.07 and 0.095);
            \draw[color=white, line width = 1mm] (0.12,1.4) arc (60:-70:0.25 and 0.73);
            \draw (0,1.5) arc (90:-90:0.25 and 0.73);
            \draw (-0.6,-0.39) ellipse (0.07 and 0.085);
            \draw [->] (-1,0.45) to (-0.8,0.6);
            \draw [->] (-0.65,-0.25) to (-0.45,-0.3);

            \node at (0.5,0.1) {\tiny{$m_i \mapsto e$}};
            \node at (-0.65,0.2) {\tiny{$m_1$}};
            \node at (-0.5,-0.6) {\tiny{$m_2$}};
            \node at (-1,0.7) {\tiny{$l_1 \mapsto e$}};
            \node at (-0.8,-0.15) {\tiny{$l_2 \mapsto e$}};
            \node at (0.4,2) {$l_i = \hat{m}_i$};
            \draw[->] (0.3,1.9) -- (0,1.6);
            \node at (4.1,2.3) {A pair of copies};
            \node at (3.9,1.9) {of $C_1$ or their};
            \node at (4.2,1.5) {longitudinal sums};
            \node at (3.8,1.1) {glue in here};
            \draw[->] (3,1.7) -- (1.6,1);
            \draw[->] (3,1.7) -- (2,-0.8);
            \node at (4.6,0) {Borromean complement,};
            \node at (4.75,-0.4) {presentation of Borromean};
            \node at (4.3,-0.8) {group $l_i = [m_1,m_2]$,};
            \node at (4.7,-1.2) {$l_1 = [m_2,m_i]$, $l_2 = [m_i, m_1]$};
            \draw [->] (3.05,0.05) -- (2.7,-0.3);

            \node at (-5,0) {\hspace{1em}};
        \end{tikzpicture}
        \caption{Picture of a Bing double. $l_i$ is a longitude of $L_i \subset L$, and drawn here as the meridian to the surgery-dual link $L^\ast$. The notation has been abbreviated: $l_1$, $l_2$, $m_1$, and $m_2$ would more precisely be written $l_1^B$, $l_2^B$, $m_1^B$, and $m_2^B$ to indicate that they label the Bing double. This model exhibits the basic commutator relations but is too simple for our application: we need ramification and internal longitudinal sums, see Figure \ref{fig:connsum}.}
        \label{fig:bingd}
    \end{figure}
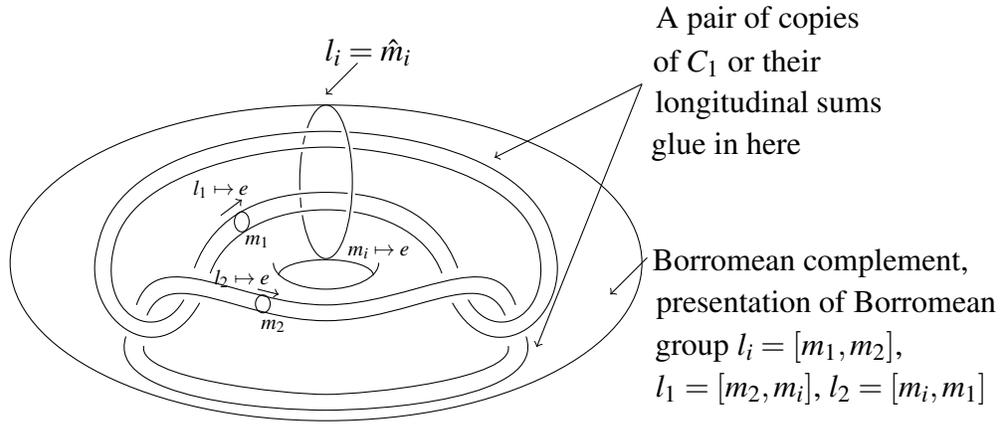

    For a component $l$ of $L^\ast$, the holonomy $\mathrm{hol}(l)$ along $l$, which must extend over $C$, lies in the identity component $\homeo_0(X)$. Using the extension of Proposition \ref{prop:x_conjugators}, write $\operatorname{hol}(l) = \prod_{i=1}^m [a_i, b_i]$ and write each $a_i$ and each $b_i$ as a product of four conjugates, two of $h$ and two of $h^{-1}$. Geometrically these 4-fold products translate into longitudinal sums of four copies of $C_1$ (and of $\lbar{C}$, at the level of bordisms). Then the sums are implanted (or attached to at the level of bordisms) to an $m$-fold ramified Bing pair, with $(a_i, b_i)$ occupying the two halves of one of the pairs $1 \leq i \leq m$.
    
    This leads us to the ramified version of Fig.\ \ref{fig:bingd} (Fig.\ \ref{fig:rambingd} below) necessary to create the model $(\lbar{\mathcal{C}}_1; \mathcal{C}_1, S^1 \times D^2)$ required to extend $\mathrm{hol}(l)$ on the meridian $\hat{m}_i$ of Figures \ref{fig:bingd} and \ref{fig:rambingd} over $\mathcal{C}_1$, the homology solid torus used to complete homology-surgery. We have used script $\mathcal{C}_1$ to denote the longitudinal sums and Bing doubling needed to pass from the base model $C_1$ to the solution to the extension problem. Similarly, we will write $\lbar{\mathcal{C}}_1$ for the corresponding 4D ssc.

    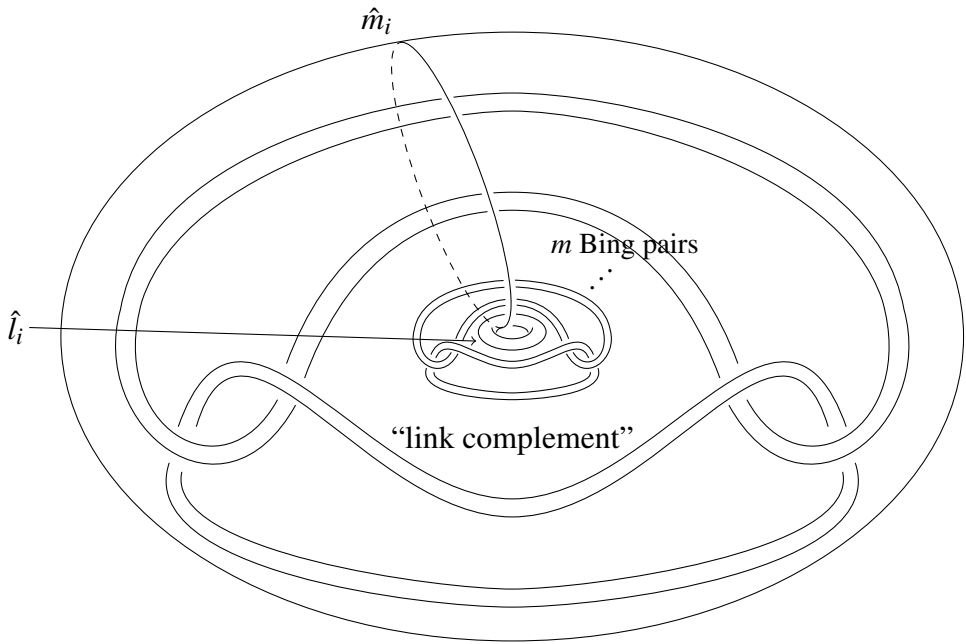
\begin{figure}[!ht]
        \centering
        \begin{tikzpicture}[scale=0.6]
            \draw (4*-1.15,4*-0.23) to [out=70,in=180] (0,2.8) to[out=0,in=110] (4*1.15,4*-0.23);
            \draw (4*-1.25,4*-0.15) to [out=70,in=180] (0,4*0.8) to [out=0,in=110] (4*1.25,4*-0.15);
            \draw (4*-1.85,4*-0.5) .. controls (4*-1.5,4*0.5) and (4*-0.8,4*-0.9) .. (0,4*-0.9) .. controls (4*0.8,4*-0.9) and (4*1.5,4*0.5) .. (4*1.85,4*-0.5);
            \draw (4*-1.75,4*-0.5) .. controls (4*-1.5,1.4) and (4*-0.8,4*-1) .. (0,4*-1) .. controls (4*0.8,4*-1) and (4*1.5,1.4) .. (4*1.75,4*-0.5);
            \draw (4*-1.3,4*-0.35) .. controls (4*-1.6,4*-0.9) and (4*-2.25,4*-0.5) .. (4*-2.05,4*0.2) .. controls (4*-1.8,4*0.95) and (4*-0.5,4*1.25) .. (0,4*1.25) .. controls (4*0.5,4*1.25) and (4*1.8,4*0.95) .. (4*2.05,4*0.2) .. controls (4*2.25,4*-0.5) and (4*1.6,4*-0.9) .. (4*1.3,4*-0.35);
            \draw (4*-1.2,4*-0.38) .. controls (4*-1.6,4*-1.1) and (4*-2.35,4*-0.5) .. (4*-2.17,4*0.15) .. controls (4*-2,4*1.1) and (4*-0.5,4*1.35) .. (0,4*1.35) .. controls (4*0.5,4*1.35) and (4*2,4*1.1) .. (4*2.17,4*0.15) .. controls (4*2.35,4*-0.5) and (4*1.6,4*-1.1) .. (4*1.2,4*-0.38);
            \draw (4*-1.9,4*-0.7) .. controls (4*-2.1,4*-1.3) and (4*-0.5,4*-1.5) .. (0,4*-1.5) .. controls (4*0.5,4*-1.5) and (4*2.1,4*-1.3) .. (4*1.9,4*-0.7);
            \draw (4*-1.83,4*-0.75) .. controls (4*-1.93,4*-1.2) and (4*-0.4,4*-1.4) .. (0,4*-1.4) .. controls (4*0.4,4*-1.4) and (4*1.93,4*-1.2) .. (4*1.83,4*-0.75);

            \draw[color=white, line width = 1mm] (-1.1,-0.1) to [out=70,in=110] (1.1,-0.1);
            \draw (-1.1,-0.1) to [out=70,in=180] (0,0.7) to[out=0,in=110] (1.1,-0.1);
            \draw (-1.25,-0.05) to [out=70,in=180] (0,0.83) to [out=0,in=110] (1.25,-0.05);
            \draw (-1.8,-0.4) .. controls (-1.5,0.3) and (-0.8,-0.55) .. (0,-0.55) .. controls (0.8,-0.55) and (1.5,0.3) .. (1.8,-0.4);
            \draw (-1.7,-0.5) .. controls (-1.5,0.1) and (-0.8,-0.7) .. (0,-0.7) .. controls (0.8,-0.7) and (1.5,0.1) .. (1.7,-0.5);
            \draw[color=white, line width = 1mm] (-1.35,-0.35) .. controls (-1.6,-0.8) and (-2.2,-0.5) .. (-2,0.2) .. controls (-1.8,0.9) and (-0.5,1.1) .. (0,1.1) .. controls (0.5,1.1) and (1.8,0.9) .. (2,0.2) .. controls (2.2,-0.5) and (1.6,-0.8) .. (1.35,-0.35);
            \draw (-1.35,-0.35) .. controls (-1.6,-0.8) and (-2.2,-0.5) .. (-2,0.2) .. controls (-1.8,0.9) and (-0.5,1.1) .. (0,1.1) .. controls (0.5,1.1) and (1.8,0.9) .. (2,0.2) .. controls (2.2,-0.5) and (1.6,-0.8) .. (1.35,-0.35);
            \draw[color=white, line width = 1mm] (-1.2,-0.38) .. controls (-1.6,-1.1) and (-2.35,-0.5) .. (-2.17,0.15) .. controls (-2,1.1) and (-0.5,1.25) .. (0,1.25) .. controls (0.5,1.25) and (2,1.1) .. (2.17,0.15) .. controls (2.35,-0.5) and (1.6,-1.1) .. (1.2,-0.38);
            \draw (-1.2,-0.38) .. controls (-1.6,-1.1) and (-2.35,-0.5) .. (-2.17,0.15) .. controls (-2,1.1) and (-0.5,1.25) .. (0,1.25) .. controls (0.5,1.25) and (2,1.1) .. (2.17,0.15) .. controls (2.35,-0.5) and (1.6,-1.1) .. (1.2,-0.38);
            \draw (-1.9,-0.7) .. controls (-2.1,-1.3) and (-0.5,-1.4) .. (0,-1.4) .. controls (0.5,-1.4) and (2.1,-1.3) .. (1.9,-0.7);
            \draw (-1.73,-0.79) .. controls (-1.8,-1.1) and (-0.4,-1.25) .. (0,-1.25) .. controls (0.4,-1.25) and (1.8,-1.1) .. (1.73,-0.79);

            \draw (0,0.1) ellipse (0.75 and 0.4);
            \draw (0.45,0.2) arc (0:-180:0.45 and 0.25);
            \draw (0.35,0.09) arc (0:180:0.35 and 0.15);
            \draw[color=white, line width = 1.5mm, rotate around={20:(-0.2,0.2)}] (-0.2,0.2) arc(-90:90:0.7 and 3.36);
            \draw[rotate around={20:(-0.2,0.2)}] (-0.2,0.2) arc(-90:90:0.7 and 3.36);
            \draw[dashed,rotate around={20:(-0.2,0.2)}] (-0.2,0.2) arc(-90:-270:0.7 and 3.36);
            \draw (0,0) ellipse (10 and 6.75);

            \node at (-3,7) {$\hat{m}_i$};
            \node at (-11,0.2) {$\hat{l}_i$};
            \draw[->] (-10.7,0.2) -- (-0.8,-0.1);)
            \node at (2.5,2) {\small{$m$ Bing pairs}};
            \node[rotate=45] at (2,1.3) {$\dots$};
            \node at (11,0) {$\hspace{0.25em}$};

            \node at (0,-2.3) {``link complement''};
        \end{tikzpicture}
        \caption{$m$-ramified Bing double.}
        \label{fig:rambingd}
    \end{figure}

    Each Bing pair consists on each side (see Figure \ref{fig:connsum}) of a 4-fold longitudinal boundary connected sum of homology solid tori $C_1$, corresponding to the multiplicities on line (\ref{eq:ab_conj_i}).

    \begin{figure}[!ht]
        \centering
        \begin{tikzpicture}[scale=1.3]
            \draw[color=white, line width = 1mm] (-1.1,-0.1) to [out=70,in=110] (1.1,-0.1);
            \draw (-1.1,-0.1) to [out=70,in=110] (1.1,-0.1);
            \draw[color=white, line width = 1mm] (-1.3,-0.1) to [out=70,in=180] (0,0.67) to [out=0,in=110] (1.3,-0.1);
            \draw (-1.25,-0.05) to [out=70,in=180] (0,0.67) to [out=0,in=110] (1.25,-0.05);
            \draw (-1.8,-0.4) .. controls (-1.6,0.2) and (-0.8,-0.4) .. (0,-0.4) .. controls (0.8,-0.4) and (1.6,0.2) .. (1.8,-0.4);
            \draw (-1.7,-0.5) .. controls (-1.5,0) and (-0.8,-0.55) .. (0,-0.55) .. controls (0.8,-0.55) and (1.5,0) .. (1.7,-0.5);
            \draw[color=white, line width = 1mm] (-1.35,-0.35) .. controls (-1.6,-0.8) and (-2.2,-0.5) .. (-2,0.2) .. controls (-1.8,0.9) and (-0.5,1.1) .. (0,1.1) .. controls (0.5,1.1) and (1.8,0.9) .. (2,0.2) .. controls (2.2,-0.5) and (1.6,-0.8) .. (1.35,-0.35);
            \draw (-1.35,-0.35) .. controls (-1.6,-0.8) and (-2.2,-0.5) .. (-2,0.2) .. controls (-1.8,0.9) and (-0.5,1.1) .. (0,1.1) .. controls (0.5,1.1) and (1.8,0.9) .. (2,0.2) .. controls (2.2,-0.5) and (1.6,-0.8) .. (1.35,-0.35);
            \draw[color=white, line width = 1mm] (-1.2,-0.38) .. controls (-1.6,-1.1) and (-2.35,-0.5) .. (-2.17,0.15) .. controls (-2,1.1) and (-0.5,1.25) .. (0,1.25) .. controls (0.5,1.25) and (2,1.1) .. (2.17,0.15) .. controls (2.35,-0.5) and (1.6,-1.1) .. (1.2,-0.38);
            \draw (-1.2,-0.38) .. controls (-1.6,-1.1) and (-2.35,-0.5) .. (-2.17,0.15) .. controls (-2,1.1) and (-0.5,1.25) .. (0,1.25) .. controls (0.5,1.25) and (2,1.1) .. (2.17,0.15) .. controls (2.35,-0.5) and (1.6,-1.1) .. (1.2,-0.38);
            \draw (-1.9,-0.7) .. controls (-2.1,-1.3) and (-0.5,-1.4) .. (0,-1.4) .. controls (0.5,-1.4) and (2.1,-1.3) .. (1.9,-0.7);
            \draw (-1.73,-0.79) .. controls (-1.8,-1.1) and (-0.4,-1.25) .. (0,-1.25) .. controls (0.4,-1.25) and (1.8,-1.1) .. (1.73,-0.79);
            \node at (0,0) {$\ast$};

            \draw (0,1.2) ellipse (0.3 and 0.4);
            \draw [->] (0.2,1.5) to [out=45,in=145] (2.5,1.5);
            \node at (0.8,2.75) {\small{blow-up of cross-section}};
            \node at (0.575,2.4) {\small{of longitudinal sums}};

            \draw  (2.5,1) rectangle (3.5,-1);
            \foreach \x in {0,...,3}
            \draw[fill=lightgray] (2.75,-0.4+0.3*\x) .. controls (2.75,-0.25+0.3*\x) and (2.85,-0.25+0.3*\x) .. (2.9,-0.25+0.3*\x) -- (3.1,-0.25+0.3*\x) .. controls (3.15,-0.25+0.3*\x) and (3.25,-0.25+0.3*\x) .. (3.25,-0.4+0.3*\x) .. controls (3.25,-0.55+0.3*\x) and (3.15,-0.55+0.3*\x) .. (3.1,-0.55+0.3*\x) -- (2.9,-0.55+0.3*\x) .. controls (2.85,-0.55+0.3*\x) and (2.75,-0.55+0.3*\x) .. (2.75,-0.4+0.3*\x);

            \draw (0,-1.3) ellipse (0.3 and 0.4);
            \draw[->] (0.2,-1.6) to [out=-35,in=230] (2.6,-1.3);
            \node at (1.2,-2.2) {\small{blow-up of cross section}};
            \node at (5,0) {\small{3 longitudinal sums}};

            \node at (-6,0) {\hspace{0.5em}};
        \end{tikzpicture}
        \caption{}
        \label{fig:connsum}
    \end{figure}

    The link complement in Fig.\ \ref{fig:rambingd} is that of a $m$-ramified Borromean ring. With the addition of the (true) relation $\hat{l} = e$ to the fundamental group of the link complement (corresponding to the triviality of the $A$-holonomy along the meridians to $L$) the group becomes freely generated by meridinal loops $a_j$ and $b_j$ to the $2m$ components in Fig.\ \ref{fig:rambingd}. This is because filling the hole in Fig.\ \ref{fig:rambingd} results in a $2m$-component unlink of solid tori. Notice the presence of the surface relation:
    \begin{equation}\label{eq:abj}
        \hat{m} = \prod_{i=1}^m[a_i,b_i]
    \end{equation}

    This is what we earlier referenced as the relationship between Bing doubling and the commutator structure on $\de S_g$, the Siefert surface for $\hat{m}$.

    Thus equation \ref{eq:abj}  can be implemented geometrically by filling the $2m$ deleted solid tori in Fig.\ \ref{fig:rambingd} with four copies each of the model $C_1$, thus realizing $\hat{m}_i$ on the boundary of a homology solid torus $\mathcal{C}_1$ with a sufficiently elaborate fundamental group that $g_i = \mathrm{hol}_A(l_i) = \mathrm{hol}_A(\hat{m}_i)$ extends:
    \begin{equation}
        \begin{tikzpicture}[anchor=base, baseline]
            \node at (-0.3,0.4) {$\pi_1(\mathcal{C}_1)$};
            \draw [->] (0.4,0.5) -- (1.3,0.5);
            \node at (2.4,0.4) {$\homeo_0(X)$};
            \node at (-0.3,-0.4) {$\hat{m}_i$};
            \draw [->] (0.3,-0.3) -- (1.7,-0.3);
            \node at (2.4,-0.4) {$\operatorname{hol}(l)$};
            \node[rotate around={90:(-0.1,0)}] at (-0.1,0) {$\in$};
            \node[rotate around={90:(2,0)}] at (0.5,2.1) {$\in$};
        \end{tikzpicture}
    \end{equation}

    To understand how this works we must explain the role on conjugation and inverse. Inverses of holonomies are implemented by applying the automorphism
    $\begin{vmatrix}
        -1 & 0 \\ 0 & - 1
    \end{vmatrix}$
    to the torus $\de C_1$. Up until now we have been careless about base points. Holonomies are really just conjugacy classes until base points are introduced. How does one describe the composition of meridinal holonomies $g_1$ and $g_2$ when two solid tori $T_1$ and $T_2$ are summed along a common longitude to form a third?
    \begin{equation}
        T = T_1\ \tikz[anchor=base,baseline]{\draw (0,0.4) -- (0,0.05) -- (0.2,0.05) -- (0.2,-0.12); \draw (0,0.25) -- (0.2,0.25) -- (0.2,0.05);}_{longitude}\ T_2
    \end{equation}

    Since in all cases the longitudes represent to the identity, the problem dimensionally reduces to asking what the pinch map $p: S^1 \ra S^1 \vee S^1$ represents given that the two wedge factors (petals) separately represent to $g_1$ and $g_2$ respectively. The answer is that we have considerable choice. By reparametrizing the fiber $X$ over the base point of $T_2$, also the base point of $S^1 \vee S^1$, by an arbitrary element of $\tau \in \homeo(X)$ we can realize any conjugacy class of the form $g_1 g_2^\tau$. This key idea, the ``simplicity trick," undergirds our ``designer extension" of arbitrary meridional holonomies.
    
    \begin{lemma}[Simplicity Lemma]\label{lm:simplicity}
        Suppose we have a collection of flat $X$-bundles
        \begin{tikzpicture}[anchor=base, baseline]
            \node at (0,1.1) {$X$};
            \draw [->] (0.2,1.2) -- (0.8,1.2);
            \node at (1.1,1.1) {$E_i$};
            \draw [->] (1.1,1) -- (1.1,0.4);
            \node at (1.1,0) {$T_i$};
            \draw [decorate,decoration={brace,amplitude=5pt}]
            (1.4,1.4) -- (1.4,-0.1);
            \draw [decorate,decoration={brace,amplitude=5pt}]
            (-0.2,-0.1) -- (-0.2,1.4);
        \end{tikzpicture}
        over homology-solid tori $T_i$ with a constant holonomy $\rho_i(l_i) = \rho(l)$ around a marked longitude $l_i \subset \de T_i$ and a variable meridianal holonomy $\rho_i(m_i) \in \homeo(X)$. If $\rho(l) = \id \in \homeo(X)$, there exists a flat bundle
        \begin{tikzpicture}[anchor=base, baseline]
            \node at (0,1.1) {$X$};
            \draw [->] (0.2,1.2) -- (0.8,1.2);
            \node at (1.1,1.1) {$E$};
            \draw [->] (1.1,1) -- (1.1,0.4);
            \node at (1.1,0) {$T$};
        \end{tikzpicture}
        over a homology-solid torus $T$, a longitudinal sum drawn from $\{T_i\}$, realizing arbitrary holonomy in the normal closure $\langle \langle \rho_i(m_i) \rangle \rangle \subset \homeo(X)$. In general if $\rho(l) \neq \id$, holonomies in the subgroup generated by elements of the form $(\rho_i(m_i))^{c_j}$ can be realized, where the $c_j$ are arbitrary elements in the centralizer $Z(\rho(l)) \subset \homeo(X)$.
    \end{lemma}
    
    \begin{proof}
        When $\rho(l) = \id$, the preceding discussion suffices. (Note: this is the case used in the present paper.) For general $\rho(l)$ the choice of reparameterizations of the fiber $X$ is necessarily restricted to those commuting with $\rho(l)$ as others do not globalize over
        \begin{tikzpicture}[anchor=base, baseline]
            \node at (0,1.1) {$X$};
            \draw [->] (0.2,1.2) -- (0.8,1.2);
            \node at (1.1,1.1) {$E_i$};
            \draw [->] (1.1,1) -- (1.1,0.4);
            \node at (1.1,0) {$T_i$};
        \end{tikzpicture}
        . In practice, it seems quite difficult to exploit Lemma \ref{lm:simplicity} when $\rho(l) \neq \id$.
    \end{proof}

    To complete the proof of Theorem \ref{thm:extend1} we need to describe $W_0$ (and $W$) and verify they are ssc.

    To build the 4D bordism $W_0$ one must take longitudinal sums of the slice-complement $\lbar{C}$, extending along an additional interval factor, the longitudinal boundary connected sums of copies of $C$. These continue to be relative $H_\ast$-cobordisms on $(S^1 \times D^2,\de)$. Schematically the cobordism $W$ is shown back in Figure \ref{fig:sscholo}, where the $\lbar{\mathcal{C}}_1$ are represented as thin hatched strips although, in reality, they are ssc.

    This completes the construction of $W$ and the extending flat connection on $V^\ast = \de_+ W$. There is a retraction $W \xrightarrow{r} V$ induced from the construction $\lbar{\mathcal{C}}_1 \xrightarrow{r} S^1 \times D^2$. To see that $W \xrightarrow{r} V$ has the structure of a semi-$s$-cobordism, note that Figure \ref{fig:sscholo} describes a product except for the hatched strips, previously denoted by $W_0$. These hatched rectangles are of the form: $(\lbar{\mathcal{C}}_1;\mathcal{C}_1,S^1 \times D^2)$. $S^1 \times D^2$ is a terminal object \footnote{Explicitly to construct the degree one map $(\mathcal{C}_1, \de) \ra (S^1 \times D^2, \de)$, recall $K = \operatorname{ker}(H_1(\de \mathcal{C}_1;\Z) \ra H_1(\mathcal{C}_1;\Z)) \cong \Z$ (a consequence of Lefschitz duality and the sequences for pairs). Consequently, we find a 1-submanifold in $\de \mathcal{C}_1$ bounding a 2-submanifold $\Delta \subset \mathcal{C}_1$. Map $\Delta$ degree one to $\ast \times D^2 \subset S^1 \times D^2$ and then extend (degree one) ($\mathcal{C}_1 \setminus \Delta, \de) \ra (S^1 \times D^2 \setminus \ast \times D^2, \de)$. Note that $\de(\mathcal{C}_1\setminus\Delta)$ is connected, making the extension degree one. More generally, higher genus handle bodies are also terminal objects in the corresponding sense.} in category of three manifolds with boundary $S^1 \times S^1$ and degree one maps, identity on boundary. Using obstruction theory we construct a relative retraction: $\lbar{\mathcal{C}_1} \xrightarrow{r} S^1 \times D^2$. While $r$ is not a deformation retraction because $\pi_1(\lbar{\mathcal{C}_1}) \not\cong \pi_1(S^1 \times D^2) \cong \Z$, it is a homology isomorphism with integral coefficients.

    Because the blocks $\lbar{\mathcal{C}}_1$ are glued to Bing doubles and therefore map trivially into $\pi_1(W)$, the inclusion $V \ra W$ induces a $\Z[\pi_1]$-homology isomorphism, $\pi_1 = \pi_1(V) \cong \pi_1(W)$ (but distinct from $\pi_1(V^\ast)$ which is much larger). This shows that $V \ra W$ is a homotopy equivalence. The point here, is that triviality of the map $\pi_1(\lbar{\mathcal{C}}_1) \ra \pi_1(W)$ means that $\pi_1(W)$ acts on lifts of $\lbar{\mathcal{C}}_1$ only by permuting blocks. It is actually a simple homotopy equivalence since, lifted to the universal cover, the inclusion of the non-product regions, i.e.\ the inclusions: $(S^1 \times D^2)_{\text{lifted copy}} \hookrightarrow \lbar{\mathcal{C}}_{\text{lifted copy}}$ are all $\Z$-homology isomorphisms (see \cite{mil66} for an introduction to Whitehead torsion).
    
    So far we have built $W$ and seen that is a ssc from its initial end $V$ to its final end $V^\ast$, and seen that the original $X$-bundle over $V$ can be cut and glued to make a flat $X$-bundle over $V^\ast$. But returning to line (\ref{eq:natmaps}) we see that the epimorphism $\alpha$ can be used, a block at a time, to extend $\pi_1(\mathcal{C}_1) \ra \homeo_0(X)$ over $\pi_1(\lbar{\mathcal{C}}_1)$. This extension, at the group level, gives the extension of the flat bundle over $\mathcal{C}_1$ to a bundle over $\lbar{\mathcal{C}}_1$ which glues up with the initial bundle on the remainder of $W$ (the penultimate $W \setminus$2-handles) to produce the $X$-bundle over our final $W$ with all claimed properties.
\end{proof}

As we go beyond $\dim(V) = 3$ the strategy which proved Theorem \ref{thm:extend1} will not work. It is true that we may use our models $\lbar{\mathcal{C}}_1$ to ``homologically'' remove relative 2-handles from the $V^\ast$ end of a cobordism $(W;V,V^\ast)$. But because they are only deleted in a homological sense, the attaching regions of the 3-handles do not reach the new $V^\ast$ and we are stuck: We do not see spherical classes in the modified $V^\ast$ to surger.

\begin{proof}[Proof of Theorem \ref{thm:extend2}]
    The plan is to follow Meigniez's proof of MT via quasi-complementary foliations up to its final step \cite{mei}. In this proof, the \emph{problem} is now not just a bundle but also a Haefliger structure on that bundle. Meigniez treats a range of smoothness from $C^1$ to $C^\infty$, and states that his results will later be extended to the bilipschitz category in a subsequent paper. The problem is \emph{solved} (the Haefliger structure made ``regular'' so that it induces a foliation) up to singularities along certain ``fissures'' $\Sigma_i$ in the total space $B$ of
    \begin{tikzpicture}[anchor=base, baseline]
        \node at (0,1.1) {$X$};
        \draw [->] (0.2,1.2) -- (0.8,1.2);
        \node at (1.1,1.1) {$B$};
        \draw [->] (1.1,1) -- (1.1,0.4);
        \node at (1.1,0) {$V$};
    \end{tikzpicture}
    .
    
    In briefest outline here is the setup we borrow from \cite{mei} with line numbers referring to that paper. 
    \begin{tikzpicture}[anchor=base, baseline]
        \node at (0,1.1) {$X$};
        \draw [->] (0.2,1.2) -- (0.8,1.2);
        \node at (1.1,1.1) {$B$};
        \draw [->] (1.1,1) -- (1.1,0.4);
        \node at (1.1,0) {$V^p$};
        \node at (1.35,0.65) {\small{$\pi$}};
    \end{tikzpicture}
    is endowed with a Haefliger structure $\Gamma$, that is a (germ of a) foliation $\mathcal{F}$ of dimension $q$ on 
    \begin{tikzpicture}[anchor=base, baseline]
        \node at (0,1.1) {$\R^q$};
        \draw [->] (0.2,1.2) -- (0.8,1.2);
        \node at (1.1,1.1) {$\tau$};
        \draw [->] (1.1,1) -- (1.1,0.4);
        \node at (1.1,0) {$B$};
        \node at (1.65,0.6) {\small{$Z$}};
        \draw[->] (1.3,0.2) to[out=45,in=-90] (1.45,0.65) to[out=90,in=-45] (1.3,1.1);
    \end{tikzpicture}
    transverse to the fibers of $\tau$, where $\tau$ is the bundle of vertical tangents in $B$. If the zero section $Z$ were transverse to $\mathcal{F}$, pulling back $\mathcal{F}$ to $B$ would flatten the original bundle 
    \begin{tikzpicture}[anchor=base, baseline]
        \node at (0,1.1) {$X$};
        \draw [->] (0.2,1.2) -- (0.8,1.2);
        \node at (1.1,1.1) {$B$};
        \draw [->] (1.1,1) -- (1.1,0.4);
        \node at (1.1,0) {$V^p$};
    \end{tikzpicture}
    . The approach is to exploit the tools Thurston introduced in the 70s to modify $\mathcal{F}$ to maintain its defining property while making it as close as possible to transverse to the zero section $Z$. Meigniez denotes the non-transverse locus by $\de \Sigma$ for boundary (total fissure core) [Prop 4.21]. Since we only address the boundary, we shorten the terminology to \emph{total fissure} and denote it by $\Sigma$, and its components \emph{fissures} by $\Sigma_i$. They enjoy various technical properties (see Definitions 3.2 through 3.8) specifying exactly how transversality between $\mathcal{F}$ and $Z$ fails on $\Sigma$. The total fissure $\Sigma \subset B$ is a $(p-2)$-dimensional submanifold of many components. There is some freedom in choosing the topology of its components $\Sigma_i$ (see his Remark 3 and our comments below). Each component $\pi$-projects 1-1 to $V^p$ (see Claim 4.15 and Property IV in proof of Theorem 1.8, page 50), but the projection restricted to all of $\Sigma$ is merely an immersion. The punch line of his proof is to resolve the transversality obstruction by doing a generalized surgery to $V$ along a sequence of codimension 2-submanifolds obtained by projecting the sequence of fissures $\{\Sigma_i\}$ to $B$. These generalized surgeries allow little control over the topology of the cobordism $W$. Our approach is to replace the projected fissure components with the models of type $\lbar{\mathcal{C}}_1$, from our collection. This allows much more control of $W$. We realize $W$ as a ssc, but the cost is we must dimensionally stabilize $W$ to match the topology of the successive fissures to that of our models. A richer supply of models could yield a stronger result.

    When the regularity is less than or equal to class $C^1$, as we have noted in section 1, the bundle $B$ can be augmented, without making essential choices, to a Haefliger structure $\Gamma$ with normal bundle the vertical tagents, $\tau$. If $V$ has boundary, $\Gamma$ is required to be regular near $\de V$. $\Gamma$ is, by definition, a foliation $\mathcal{F}$ on $\tau$, which is transverse to fibers. $\Gamma$ is called regular if it is also transverse to $Z(B)$, the zero-section of $\tau$. Using a fine triangulation $K$ of the total space of $\tau$, and many ingenious adaptations of Thurston's constructions \cite{thur76}, control of $\mathcal{F}$ w.r.t.\ $K$ is gradually obtained. Although it is not possible to homotope $\Gamma$ to full regularity i.e.\ to make the Haefliger structure everywhere transverse to the fiber, this is achieved in the complement of the total fissure $\Sigma$ a dimension $p-2$ embedded submanifold $\Sigma \subset B$, $p = \dim(V)$. (See Fig.\ 6 of in \cite{mei} where our $\Sigma$ is called the $(\text{[fissure core]} \cap M \times 1) = \de \Sigma$. We have no need of the part of $\Sigma$ in the interior of the cobordisms $W$ so we have simplified the notation slightly, calling the ``projected boundary of the fissure'' simply the ``fissure''. Also in \cite{mei} our $B$ is written $M$.) From the simplicial, actually ``prismatic,'' structure $K$, it is deduced that 1) the components of $\Sigma$ all have small diameter 2) the topology of each component can be separately controlled to be either $S^1 \times S^{p-3}$ or the $(p-2)$-torus $(S^1)^{p-2}$, and 3) individually each component of $\Sigma$ $\pi$-projects 1-1 as embedded submanifolds of $V$, which we still call the fissures, $\pi\Sigma_i$, and $\Sigma \coloneqq \bigcup_{i=1}^l \Sigma_i$.

    If the fissures $\pi\Sigma_i$ were pairwise disjoint, the replacement device we describe next would prove a stronger version of Theorem \ref{thm:extend2} without the stabilizing manifold $Q$. The replacement models (see our section 3) we know how to build require a fissure $\pi\Sigma_i$ to have spherical factors i.e.\ have the form $M \times S^k$ for some closed manifold $M$ and $1 \leq k \leq p-3$. It is possible that with more powerful machinery for building replacement models with more general boundary conditions, Thm.\ \ref{thm:extend2} might still be proven without stabilization.

    A priori, it may seem that Meigniez's fissure models are just what is needed as each contains a circle factor. The difficulty comes from the apparently unavoidable crossing of (projected) fissures $\pi \Sigma_i \cap \pi \Sigma_j \neq \varnothing \subset V$, and the structure of his final induction. Although \cite{mei} and this paper heal the initial fissure $\pi\Sigma_1$ ($\pi\Sigma_l$ in \cite{mei}) by replacing its neighborhood with different models, in our case $\mathcal{C}_1$, in both cases there is a degree one map (called ``a'' in the last paragraphs before the end of section 5 \cite{mei}) from the replacement to what is replaced. Since we will not reference the bundle $B$ again, let us simplify notation and drop the $\pi$ from $\pi\Sigma$. A generic argument based on Thom transversality for which Meigniez credits Poenaru (see Lemma 3.12) allows $\Sigma_i$, $i > 1$, to be ``pulled'' through the replacement $\mathcal{C}_1$, and become the new (projected) fissures for a cobordant problem over a modified $V$, called $V^\ast$. These new fissures, $\Sigma_2^\pr,\dots, \Sigma_l^\pr$, are still individually embedded in $V^\ast$, still small when measured by the natural retraction back to $V$, but their topology may now be quite complicated, with no circle factor present. Since the replacement model in \cite{mei} is factor preserving, $\Sigma_i \times D^2$ is replaced with $\Sigma_i \times S_g$, where $S_g$ is the genus $g$ surface with one boundary component. The loss of control on the topology of $\Sigma_i^\pr$, $i \geq 2$, presents no problem for him, the old $S_g$ model still suffices, and indeed he is not trying to control the topology of the cobordism $W$ anyway. Because our replacement models $\mathcal{C}_1^p$ require an $S^k$ factor, $1 \leq k \leq p-2$, we resort to crossing with (any) such factor to augment $\Sigma_j^\pr \subset V^\ast$ to $\Sigma_j^\pr \times S^k \subset V^\ast \times S^k$, $2 \leq j \leq l$, a stabilized problem.

    Proceeding then in $l$ steps, healing first $\Sigma_1 \subset V$ to get $V^\ast$, healing $\Sigma_2^\pr \times S^k \subset V^\ast \times S^k$ to get $\Sigma_3^{\pr} \times S^{k^\pr} \subset V^{\ast\ast} \times S^{k^\pr}$ we end up building, after stabilizing by a product of $l-1$ spheres, a composition of $l$ ssc $W$. $W$ is seen to be a ssc by composing retractions. The result is an ssc from the original problem stabilized by $Q$, a product of $(l-1)$-spheres, to a solution over $V$\raisebox{5pt}{$\scriptstyle{\underbrace{\ast \cdots \ast}_{l\text{-times}}}$}.

    The reason stabilization is \emph{not} necessary in Thm.\ \ref{thm:extend2} when $p = \dim(V) = 3$, is in that case all $\Sigma_i \subset V$ are circles and project pairwise disjointly by general position.

    Having described the relation with \cite{mei} here are some further details of our construction.

    First, for concreteness, assume we choose $S^1 \times S^{p-3}$ for the topology of $\Sigma_1$. The initial replacement of $\Sigma_1 \times D^2$ with our model $\mathcal{C}_1^k$ is via the cobordism (which we built as a slice complement) $W_1 \coloneqq \lbar{\mathcal{C}}_1^k$. $\Sigma_2^\pr$ is some $(p-2)$-submanifold of $V^\ast$ but we do not know its topology. Crossing everything now with $S^{k_1}$, the new fissure is $\Sigma_2^\pr \times S^{k_1}$, its neighborhood is $\Sigma_2^\pr \times S^{k_1} \times D^2$ with boundary $\Sigma_2^\pr \times S^{k_1} \times S^1$. We have in our stockpile a model $\mathcal{C}_1^{k_1+1}$, with a ssc $\lbar{\mathcal{C}}_1^{k_1+1}$ to $D^{k_1+1} \times S^1$. Now cross this model with $\Sigma_2^\pr$ and glue the result to $V^\ast$ times $S^{k_1}$ to form the next chamber, $W_2$. So far we have built
    \begin{equation}
        W_1 \times S^{k_1} \cup W_2
    \end{equation}

    Sequentially build $W$:
    \begin{equation}
        W = W_1 \times S^{k_1} \times \cdots \times S^{k_{l-1}} \cup W_2 \times S^{k_2 \times \cdots \times k_{l-1}} \cup \cdots \cup W_{l-1} \times S^{k_{l-1}} \cup W_l
    \end{equation}

    Each chamber is a ssc by the argument used in Theorem \ref{thm:extend1}, so the composition $W$ is as well.

    To complete the proof we must check that the replacements $\mathcal{C}_1^k$ can be engineered to solve the representation extension problem. In the stabilization direction, $S^k$, even if we chose $k = 1$, the holonomy is trivial owing to the \emph{product} form of stabilization.

    A nice feature of \cite{mei} is that he allows us to choose arbitrarily \emph{any} nontrivial fiber $X$ holonomy $\phi$ (see his Theorem A\textprime) around the normal circle to $\Sigma_i$. So at first it looks like the technology of bounded simplicity does not need to be invoked. However, there could be a small gap. In the proofs of \cite{mei}, ``any'' means any nontrivial $C^1$-diffeomorphism in the identity component $\mathrm{Diff}^1_0(X)$. However, our simplest models $C_1^k$ produce holonomy $h$ (as in Thm.\ \ref{thm:extend1} proof), $h \in \mathrm{bilipschitz}_0(X)$.

    In private communications, I have learned Meigniez's result does hold in the bilipschitz category (details to appear), so using this we can avoid the more complicated replacement used in the proof of Theorem \ref{thm:extend1} and build $\mathcal{C}_1$ and $\lbar{\mathcal{C}}_1$ as a straightforward Bing double of a pair $((\lbar{C}_1, \lbar{C}_1^\pr);(C_1,C_1^\pr))$, and demand from \cite{mei} Theorem A\textprime of section 3 that the fissures be built with normal holonomy $= [\alpha_{C_1},\alpha_{C_1^\pr}]$, the commutator of two models with different choice of $h$ so as not to commute, see our section 3 on dynamics.
    
    However, as an alternative to relying on Meigniez's recent extension to the bilipschitz category, the next paragraph does supply the details needed to pass from a model with normal holonomy in $\operatorname{bilipschitz}(X)$ or even in $\homeo(X)$ to fissures with holonomy in $C^1$.

    We employ the same tricks, longitudinal sums and iterated Bing doubles, used to prove Theorem \ref{thm:extend1}. The failure of transversality across each fissure $\Sigma_i$ is confined in the vertical (X) coordinate to a ball $D_i^q \subset X$, which we write simply as $D^q$. What is necessary is to express a (any) nontrivial element $\te \neq \id \in \mathrm{Diff}_0^1(D^q,\de)$ as a product of commutators of products of conjugates of the bilipschitz homeomorphism $h: D^q \ra D^q, \id\big\vert_{\de D^q}$, as in the proof of Thm.\ \ref{thm:extend1}, which acts on that ball $D^q \subset X$ via a faithful action of $\tld{\Delta}(2,3,7)$.
    \[
        \te = \prod_{k=1}^K\left[\prod_{j=1}^{J_k} (h^{\pm})^{a_{kj}}, \prod_{j=1}^{J_k} (h^{\pm})^{b_{kj}}\right] \text{ for } a_{kj},b_{kj} \in \homeo_0(D^q,\de),\ 1 \leq k \leq K
    \]

    This formula models the construction of an composition $(\lbar{\mathcal{C}}_1, \mathcal{C}_1)$ of the simple model $(\lbar{C}_1, C_1)$ along the lines summarized in Figures \ref{fig:rambingd} and \ref{fig:connsum}. The model $(\lbar{\mathcal{C}}_1, \mathcal{C}_1)$ solves the stabilize cleft-replacement problem allowing the cobordisms $W_i$ to be built.
\end{proof}

\begin{proof}[Proof of Theorem \ref{thm:four}]
    The proof is entirely formal. It uses the Mather-Thurston theorem \cite{mc}, including the relative version where the bundle is \emph{already} flat near the boundary of the base, and an induction over a handle decomposition of $V$.

    Let $\mathcal{H}$ be a relative handle decomposition of $V$. It is trivial to reduce to the discrete structure group $\homeo^\delta(X)$ over $\de V \cup \mathcal{H}_0 \cup \mathcal{H}_1$, over the boundary and the 0- and 1-handles. We call this reduction a \emph{solution} over the 1-handles. Now consider a 2-handle $h_2$. By M-T, on $h_2$ relative to the solution over its attaching region\footnote{I use here the convention that the attaching region of a handle $h$ is written $\de_+h$ and the co-attaching region $\de_-(h)$.} $\de_+(h_2)$, the bundle over $h_2$ is bordant, rel $\de_-(h_2)$, to a \emph{solution} over a surface $k_2$, and a small thickening of $k_2$, where the structure group will also be discrete. Observe that $(h_2,\de h_2) \simeq (D^2, \de D^2) \times D^{p-2}$ enjoys the universal extension property of accepting degree one maps:
    \begin{equation}
        \begin{tikzpicture}[anchor=base, baseline]
            \node at (-0.3,0.5) {$\Sigma^2$};
            \draw [dashed,->] (0.3,0.6) -- (1.9,0.6);
            \node at (2.4,0.5) {$D^2$};
            \node at (-0.3,-0.5) {$\de \Sigma$};
            \draw [->] (0.3,-0.4) -- (1.9,-0.4);
            \node at (2.4,-0.5) {$\de D^2$};
            \node at (-0.3,0.05) {$\hookuparrow$};
            \node at (2.3,0.05) {$\hookuparrow$};
            \node at (1.1,0.7) {\small{deg 1}};
            \node at (1.1,-0.3) {\small{deg 1}};
            \node at (4.3,-0.5) {,\quad $D^2 = \mathrm{core}(h_2)$};
            \node at (-4,0) {\hspace{0.25em}};
        \end{tikzpicture}
    \end{equation}

    Apply this property to the relative surface $k_2$. This retracts $k_2$, where the solution has so far been built, from $\de V \cup \text{0-handles} \cup \text{1-handles} \cup \EuScript{N}(\cup k_2\text{'s})$, back to $\de V \cup$0, 1, 2-handles.

    Now if $h_3$ is a 3-handle, its attaching region now receives maps from a union of solutions $k_2$ (one $k_2$ for every facet of the attaching map of $h_3$ to the lower index handles). In fact, the just constructed cobordisms over the (cores of the) 2-handles fit together with $\mathrm{core}(h_3)$ to define a flattening problem over $\de V \cup$0, 1, 2, 3-handles (solvable again by MT). Its solution $k_3$ again has a universal property:
    \[
        \begin{tikzpicture}[anchor=base, baseline]
            \node at (-0.3,0.5) {$M^3$};
            \draw [dashed,->] (0.3,0.6) -- (1.9,0.6);
            \node at (2.4,0.5) {$D^3$};
            \node at (-0.3,-0.5) {$\de M^3$};
            \draw [->] (0.3,-0.4) -- (1.9,-0.4);
            \node at (2.4,-0.5) {$\de D^3$};
            \node at (-0.3,0.05) {$\hookuparrow$};
            \node at (2.3,0.05) {$\hookuparrow$};
            \node at (1.1,0.7) {\small{deg 1}};
            \node at (1.1,-0.3) {\small{deg 1}};
        \end{tikzpicture},
    \]
    and so maps degree 1 to $h_3$ (see the schematic in Figure \ref{fig:hmaps} which has been reduced by one dimension for legibility).

    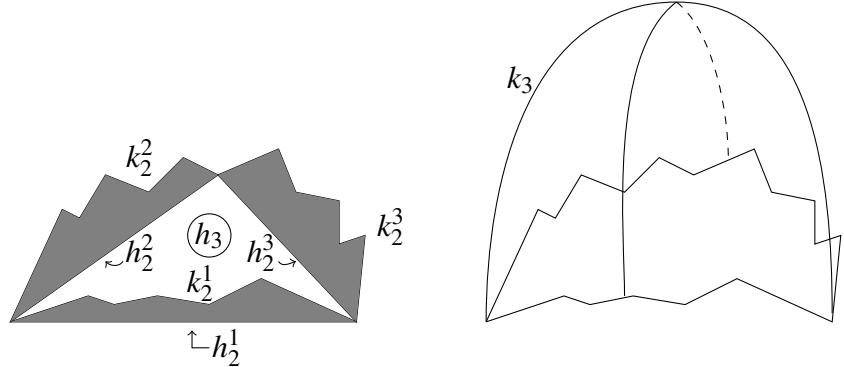
\begin{figure}[!ht]
        \centering
        \begin{tikzpicture}[scale=1.15]
            \draw (-2.2,-1.5) -- (0.2,0.2) -- (1.8,-1.5) -- cycle;
            \draw (-2.2,-1.5) -- (-1.6,-0.2) -- (-1.4,-0.3) -- (-1.1,0.2) -- (-0.6,0) -- (-0.2,0.4) -- (0.2,0.2) -- (0.9,0.5) -- (1.1,0) -- (1.6,-0.1) -- (1.6,-0.6) -- (1.9,-0.5) -- (1.8,-1.5) -- (0.7,-1) -- (0.1,-1.3) -- (-0.5,-1.2) -- (-1,-1.3) -- (-1.3,-1.2) -- (-2.2,-1.5);
            \fill[color = gray] (-2.2,-1.5) -- (-1.6,-0.2) -- (-1.4,-0.3) -- (-1.1,0.2) -- (-0.6,0) -- (-0.2,0.4) -- (0.2,0.2)-- (0.9,0.5) -- (1.1,0) -- (1.6,-0.1) -- (1.6,-0.6) -- (1.9,-0.5) -- (1.8,-1.5) -- (0.2,0.2) -- cycle;
            \fill[color=gray] (1.8,-1.5) -- (0.7,-1) -- (0.1,-1.3) -- (-0.5,-1.2) -- (-1,-1.3) -- (-1.3,-1.2) -- (-2.2,-1.5) -- cycle;
            \node at (-0.7,0.4) {$k_2^2$};
            \node at (2.2,-0.4) {$k_2^3$};
            \node at (0,-1.05) {$k_2^1$};
            \node at (-0.7,-0.7) {$h_2^2$};
            \draw[->] (-0.9,-0.8) to[out=240,in=-45] (-1.1,-0.8);
            \node (v1) at (0.1,-0.5) {$h_3$};
			\draw  (v1) circle (0.25);
            \node at (0.7,-0.7) {$h_2^3$};
            \draw[->] (0.9,-0.8) to[out=-60,in=200] (1.1,-0.8);

            \draw (5.5-2.2,-1.5) -- (5.5-1.6,-0.2) -- (5.5-1.4,-0.3) -- (5.5-1.1,0.2) -- (5.5-0.6,0) -- (5.5-0.2,0.4) -- (5.7,0.2) -- (5.5+0.9,0.5) -- (6.6,0) -- (7.1,-0.1) -- (7.1,-0.6) -- (5.5+1.9,-0.5) -- (5.5+1.8,-1.5) -- (6.2,-1) -- (5.6,-1.3) -- (5.5-0.5,-1.2) -- (5.5-1,-1.3) -- (5.5-1.3,-1.2) -- (5.5-2.2,-1.5);
            \draw (5.5,2.2) to[out=180,in=90] (3.3,-1.5);
            \draw (5.5,2.2) to [out=0,in=90] (7.3,-1.4);
            \draw (5.5,2.2) .. controls (4.7,1.6) and (4.9,-0.5) .. (4.9,-1.2);
            \draw[dashed] (5.5,2.2) .. controls (5.9,1.9) and (6.1,1) .. (6.1,0.37);
            \node at (2.9,2.6) {$k_3$ is the solution to the flattening problem over the 3-handle $h_3$};
			\node at (0.3,-1.8) {$h_2^1$};
            \draw[->] (0.1,-1.8) -- (-0.1,-1.8) -- (-0.1,-1.6);
            \node at (3.7,1.3) {$k_3$};
        \end{tikzpicture}
        \caption{Bordisms $b_3^i$ between $h_2^i$ and $k_2^i$ are shaded.}\label{fig:hmaps}
    \end{figure}

    The new problem $h_3 \cup (\cup_i b_3^i)$ is bordent via $b_4$ to the solution $k_3$, i.e.\ a reduction to $\homeo^\delta(X)$. So far we have described how to construct a cobordism, with retraction, over the 2 and 3-handles.

    Proceed in the same way, handle by handle, to build $W$ as a union of the bordisms $b_d$. The desired retraction to $V$ is a composition of a union of the universal degree one maps to the handle-cores of $\mathcal{H}$.
\end{proof}

Next we turn to Theorem \ref{thm:refine1}\textprime, whose initial data include a more specialized structure over $\de V$. The proof of Theorem \ref{thm:refine1}\textprime\ has an ad hoc, cut and glue, aspect that does not obviously bring a bundle cobordism along with it. However, at the end it can be ``backfilled'' using elementary homotopy theory.

\begin{proof}[Proof of Theorem \ref{thm:refine1}\textprime]
    Finally we use our second set of models $C_2$ and $\mathcal{C}_2$. The starting point is the presumed extension of the flat $g$-connection over $V^\pr$, $\de V^\pr = \Sigma$. As is now familiar, there is some framed linke $L^\pr \subset V^\pr$ so that $V = V^\pr \slash \mathcal{S}(L^\pr)$. Because we need to propagate the connection $A$ across the surgeries, they should not be standard but will be based on these homology models.

    From line (\ref{eq:natmaps}) by using the composition:
    \begin{equation}\label{eq:secondm}
        \begin{split}
            & \pi_1(C_2) \ra \pi_1(P_2) \overset{\cong}{\ra} \operatorname{BI} \\
            & \hat{m} \mapsto \alpha(\hat{m}),\text{ order } \alpha(\hat{m}) > 2
        \end{split}
    \end{equation}
    where BI denotes the binomial icosahedral group, the fundamental group of the Poincar\'{e} homology sphere \cite{ks79}.

    It is well known that every compact, simply-connected, semi-simple Lie group $G$ contains a copy of $\operatorname{SU}(2)$ (see \cite{hum}), thus for any such $G$ line (\ref{eq:secondm}) extends to line (\ref{eq:choice1})
    \begin{equation}\label{eq:choice1}
        \pi_1(C_2) \ra \pi_1(P_2) \xrightarrow{\cong} \operatorname{BI} \overset{\mathrm{inj}}{\hookrightarrow} \operatorname{SU}(2) \overset{\mathrm{inj}}{\hookrightarrow} G
    \end{equation}
    and $\alpha(\hat{m})$ maps to a non-central element $g \in G$.

    We need to find an inclusion, $\operatorname{SU}(2) \hookrightarrow G$, which normally generates $G$, so that $g$, above, will also normally generate $G$. $G$ admits a decomposition, canonical up to permutation, into a direct sum of $J$ simple Lie groups. Pick an $\operatorname{SU}(2)$ subgroup for each factor of $G$ and let $\Delta$ be the diagonal within Cartesian product of these subgroups:
    \begin{equation}
        \operatorname{SU}(2) \cong \Delta \subset \prod_{j=1}^J \operatorname{SU}(2)_j \subset G
    \end{equation}

    The normal closure of $\Delta$ in $G$ does not lie in any product factor, as $\Delta$ does not. This means that the normal closure $\langle \Delta \rangle_G = G$.

    But $\alpha(\hat{m})$ normally generates any $\operatorname{SU}(2)$ in which it lies (since it is not central). Thus under the composition
    \[
        \begin{tikzpicture}[anchor=base, baseline]
            \node at (0,0.4) {$\pi_1(C_2) \ra \pi_1(P_2) \ra BI \ra \Delta \ra G$};
            \node at (-2.3,-0.4) {$\hat{m}$};
            \draw [->] (-2,-0.3) -- (2.4,-0.3);
            \draw (-2,-0.2) -- (-2,-0.4);
            \node at (2.7,-0.4) {$g$};
            \node[rotate around={90:(-0.1,0)}] at (-2.1,0) {$\in$};
            \node[rotate around={90:(2,0)}] at (0.8,2.1) {$\in$};
        \end{tikzpicture}
    \]
    $\hat{m}$ maps to $g$ normally generating $G$.

    Now the simplicity trick, essential to Theorem \ref{thm:extend1}, allows us to build $\mathcal{C}_2$ from $C_2$ along the pattern illustrated in Figures \ref{fig:rambingd} and \ref{fig:connsum}. Since $G$ is a semi-simple Lie group, it is perfect so the general element $f \in G$ can be written first as a product of commuters:
    \begin{equation}
        f = [a_m,b_m] \cdots [a_1,b_1]
    \end{equation}
    and then $a$'s (and $b$'s) can be further written as products of conjugates of $g$ and $g^{-1}$. This leads the of Fig.\ \ref{fig:rambingd} (with ramification $=m$) and a pattern similar to Fig.\ \ref{fig:connsum} with the multiplicities of longitudinal boundary connected sums being the number of conjugates of $g$ and $g^{-1}$ needed to express $a_i$, $b_i$, $1 \leq i \leq m$.

    These representation extensions ensure the flat extension of $A$ over $V^\ast$. The semi-$s$-cobordisms $W_0$ and $W$ are now constructed precisely as in the proof of Theorem \ref{thm:extend1}. In this case all the group theory has been done inside $G$; it was not necessary to make an enlargement to $\homeo_0(G)$.

    The construction so far has built a bundle only over $\de W = V \cup \Sigma \times I \cup V^\ast$, flat on $\Sigma \times I \cup V^\ast$, but we do not yet have a $G$-bundle over $W$. The hypothesis that $\pi_1(G) \cong 0$ enables us to fill in the bundle over $W$ by obstruction theory. The possible obstructions lie in:
    \begin{equation}
        H_{k+1}(W, \de W; \pi_k(\mathrm{BG})) \cong H_k(V, \Sigma; \pi_{k-1}(G))
    \end{equation}
    using that $W$ is ssc and $\Omega \mathrm{BG} \simeq G$. But since $\dim V = 3$, we may restrict to $k \leq 3$, a range in which the coefficient groups vanish. Using the fact that for finite dimensional Lie groups, $\pi_2(G) \cong 0$, the bundle over $\de W$ extends over $W$.
\end{proof}

\begin{proof}[Proof of Theorem \ref{thm:refine2}]
    Following \cite{mil58} and \cite{wood}, up to sign, Euler class $\chi(B)$ may be written as:
    \begin{equation}\label{eq:gammas}
        \abs{\chi(B)} = \abs{p^{-1} \phi_g(\Gamma_1, \dots, \Gamma_{2g})} \in \Z = \operatorname{ker}(\tld{\operatorname{PSL}}(2,\R) \ra \operatorname{PSL}(2,\R))
    \end{equation}
    where we are using part of the fibration sequence for a universal covering space
    \[
        1 \ra \Z \cong \pi_0(\text{fiber}) \xrightarrow{p} \tld{\operatorname{PSL}}(2,\R) \ra \operatorname{PSL}(2,\R) \ra 1
    \]
    and $\Gamma_1, \dots, \Gamma_{2g}$ are (arbitrary) lifts to $\tld{\operatorname{PSL}}(2,\R)$ of
    \[
        \gamma_1 = \rho(a_1),\ \gamma_2 = \rho(b_2), \dots, \gamma_{2g} = \rho(b_g) \in \operatorname{PSL}(2,\R)
    \]
    images of the usual generators, $\{a_1,b_1,\dots,a_g,b_g\}$, for $\pi_1(\Sigma_g)$ under some representation $\rho: \pi_1(\Sigma_g) \ra \operatorname{PSL}(2,\R)$ inducing $B$. These generators satisfy the familiar relation indicated on line (\ref{eq:famrel}).
    \begin{equation}\label{eq:famrel}
        \phi_g \coloneqq [a_1,b_1] \dots [a_g, b_g] = 1
    \end{equation}

    The idea is that in the Lie algebra $\operatorname{sl}(2,\R)$ we have generators and bracket structure constants:
    \begin{equation}
        [E,F]_{\text{Lie}} = H,\ [H,E]_{\text{Lie}} = 2E,\ [H,F]_{\text{Lie}} = -2F
    \end{equation}
    where $E$ and $F$ are boosts and $H$ infinitesimally generates rotation.

    Exponentiating and applying the group theoretic bracket we find
    \begin{equation}\label{eq:haserr}
        [e^{\epsilon E}, e^{\epsilon F}] = e^{\epsilon^2H} + O(\epsilon^3)
    \end{equation}
    where $O(\epsilon^3)$ is an error term, $\frac{O(\epsilon^3)}{\epsilon^3} \leq \mathrm{constant}$.

    This suggests attempting to build $\rho$ according to the formula
    \begin{equation}\label{eq:attpf}
        \rho(a_i) = e^{\epsilon E},\ \rho(b_i) = e^{\epsilon F},\ 1 \leq i \leq g
    \end{equation}

    This is close, but will not actually obey the relation (\ref{eq:famrel}) because of the error term on line \ref{eq:haserr}, and a less serious integrality issue. So this choice of $\rho$ is \emph{not} a representation.

    We give three (different) ways to resolve this discrepancy, because in different contexts each can be useful.

    1. The simplest is to recall that the elliptic elements of $\operatorname{PSL}(2,\R)$, i.e.\ those lying in compact 1-parameter subgroups, form an open subset of $\operatorname{PSL}(2,\R)$ (in both smooth and Zariski topologies), and that each of these circle-subgroups is conjugate to the subgroup of rotations. The small error, due to the Campbell-Baker-Hausdorff formula on (\ref{eq:haserr}) means that the group commutator $c = [e^{\epsilon E}, e^{\epsilon F}]$ will not lie in $\operatorname{SO}(2) \subset \operatorname{PSL}(2,\R)$ but in a nearly conjugate subgroup $\operatorname{SO}(2)^\gamma$, $\gamma \in \EuScript{N}_{O(\epsilon^3)}(\id)$, an element within an $O(\epsilon^3)$-neighborhood of $\id \in \operatorname{PSL}(2,\R)$. Setting $\rho(a_i) = \gamma e^{\epsilon E}\gamma^{-1}$ and $\rho(b_i) = \gamma e^{\epsilon F}\gamma^{-1}$ now produces a commutator in $\operatorname{SO}(2)$. By continuity $\epsilon$ can be adjusted (prior to the conjugating) to make the commutator of finite order in any desired range. This allows us to satisfy the surface relation (\ref{eq:famrel}).

    2. A second method is to use the well known fact \cite{wood} that every element $c$ of $\operatorname{PSL}(2,\R)$ can be written as a single commutator: $c = a_0^{-1}b_0^{-1}a_0b_0$, and if $\lVert c \rVert = O(\epsilon^2)$, $\lVert a_0 \rVert, \lVert b_0 \rVert = O(\epsilon)$. Given this fact, we can just ignore the $O(\epsilon^3)$ deviation of $c$ from the $\operatorname{SO}(2)$ with infinitesimal generator $H$ in $\operatorname{PSL}(2,\R)$ and let the $p$th powers of $c$ evolve along its elliptic subgroup. At the end we ``patch up'' the accumulated discrepancy $c^p \neq \id$ by writing $c^p = \rho(a_0^{-1})\rho(b_0^{-1})\rho(a_o)\rho(b_0)$ over one additional genus, to achieve a well-defined representation. Actually this is slightly too simple since we have lost norm-control of $a_0$ and $b_0$. But this is easily restored by instead devoting $\leq \frac{1}{\epsilon^2}$ genera to the patching by choosing $k$th root $(c^p)^{\frac{1}{k}}$ and $(c^p)^{\frac{1}{k}} = \rho(a_i^{-1})\rho(b_i^{-1})\rho(a_i)\rho(b_i)$, $k + 1 \leq i \leq 0$.

    3. A final method is in the spirit of the inverse function theorem. Of course the differential of the Lie bracket is identically zero, but its 2-jet is very useful
    \begin{equation}
        [rH + sE + tF, r^\pr H + s^\pr E + t^\pr F] = (st^\pr - ts^\pr)H + 2(rs^\pr - r^\pr s)E + 2(rt^\pr - r^\pr t)F
    \end{equation}

    At $(r,s,t,r^\pr,s^\pr,t^\pr) = (0,0,0,0,0,0)$, second order independent variations (of both signs) of the coefficient of $H(E,F)$ are generated by the variables $(s,t^\pr)((r,s^\pr),(r^\pr,t))$. This freedom to move independently in all three directions of $\operatorname{sl}(2,\R)$ allows the construction of a Cauchy sequence in $\EuScript{N}_\epsilon \times \EuScript{N}_\epsilon$ whose limit is mapped by bracket to any desired point in $\EuScript{N}_{\epsilon^2}$, $\EuScript{N}_r$ the radius $r$ neighborhood of the identity in $\operatorname{SL}(2,\R)$. As with the usual proof of the inverse function theorem, exact solutions may be found near approximate solutions. This approach is the most direct in that the approximate representation $\rho$ inspired by the relations in the Lie algebra is simply perturbed to an exact group representation.

    Now return to line (\ref{eq:attpf}) and modify our provisional definition of the representation $\rho$ by setting
    \begin{equation}
        \rho(a_i) = e^{\epsilon E^\pr},\ \rho(b_i) = e^{\epsilon F^\pr},\ 1 \leq i \leq g
    \end{equation}
    where $g = \lfloor\frac{2\pi\chi(B)}{\epsilon^2}\rfloor + O(\epsilon^{-1})$, and with $E^\pr$ and $F^\pr$ defined by one of the three methods above. Referring to line (\ref{eq:gammas}), choosing $g$ of this size enables the lift to $\tld{\operatorname{PSL}}(2,\R)$ to be controlled to achieve
    \begin{equation}
        \abs{\chi(B)} = \big|p^{-1}\phi_{g+1}\big(\underbrace{\tld{e}^{\epsilon E^\pr},\tld{e}^{\epsilon F^\pr},\dots,\tld{e}^{\epsilon E^\pr},\tld{e}^{\epsilon F^\pr}}_{g \text{ pairs}}, \tld{v}, \tld{w}\big)\big|
    \end{equation}
    where $\tld{\hphantom{m}}$ indicates a lift to $\tld{\operatorname{PSL}}(2,\R)$. That is, approximately $\frac{2\pi\chi(B)}{\epsilon^2}$ commutators suffice to produce a translation in $\tld{\operatorname{PSL}}(2,\R)$ of length $2\pi\chi(B)$, where each individual commutator produces a translation of approximately $\epsilon^2$. This completes the proof of Theorem \ref{thm:refine2}.
\end{proof}

\bibliography{references}   

\end{document}